\documentclass [12pt,a4paper,reqno]{amsart}
\usepackage{amssymb}

 \input xy
\xyoption{all}
\usepackage{pstricks}

\usepackage{thmtools}
\declaretheoremstyle[notefont=\bfseries,notebraces={}{},%
    headpunct={},postheadspace=1em]{mystyle}
\declaretheorem[style=mystyle,numbered=no,name=Problem]{prob-hand}
\def\dispace{\setlength{\itemsep}{2pt}}

\newcommand\isoto{\xrightarrow{
   \,\smash{\raisebox{-0.45ex}{\ensuremath{\scriptstyle\sim}}}\,}}

\newcommand{\onto}[1]{\;{\count255=0 \loop \relbar\joinrel
    \advance\count255 by1
    \ifnum\count255<#1 \repeat \twoheadrightarrow}\;}
\newcommand{\Onto}{\mathrel  - \joinrel \twoheadrightarrow}

\def\Dir{\Rightarrow}

\def\tlu{\widetilde u}
\def\tl0{\widetilde 0}

\def\R{\mathbb R}
\def\T{\mathbb T}

\newcommand{\etype}[1]{\renewcommand{\labelenumi}{(#1{enumi})}}
\def\eroman{\etype{\roman} \dispace}
\def\ealph{\etype{\alph}\dispace}

\def\pSkip{\vskip 1.5mm \noindent}
\newcommand{\ds}[1]{\ {#1} \ }
\newcommand{\dss}[1]{\quad {#1} \quad }

\def\semiring0{semiring$^{\dagger}$}

\def\dmax{d_{\max}}

\def\sm{\setminus}
\def\00{ \{ 0 \}}
\def\o00{\overline{\00}}
\def\onto{\twoheadrightarrow}

\def\|{\ds |}

\def\vrp{\varphi}
\def\ivrp{\vrp^{-1}}

\def\Fix{{\operatorname{Fix}}}
\def\Compl{{\operatorname{Compl}}}
\def\Arch{{\operatorname{Arch}}}

\def\X1{X_1}
\def\Y1{Y_1}

\def\brV{\overline{V}}
\def\brR{\overline{R}}
\def\brS{\overline{S}}
\def\brD{\overline{D}}
\def\brC{\overline{C}}
\def\bra{\bar a}

\def\brd{\bar d}
\def\brx{\bar x}
\def\bry{\bar y}

\def\tT{\mathcal T}

\def\tG{\mathcal G}

\def\conv{\operatorname{conv}}

\def\N{\mathbb N}

\def\mfo{\mathfrak o}

\def\Ddw{D^{\downarrow}}

\def\X1{X_1}
\def\Y1{Y_1}

\def\al{\alpha}

\def\pal{\; \underset{\al}{+} \;}
\def\pal{+_\al}
\def\cal{C_\al}

\newtheorem{thm}{Theorem} [section]
\newtheorem*{thm*}{Theorem}
\newtheorem{cor}[thm]{Corollary}
\newtheorem{lem}[thm]{Lemma}

\newtheorem{prop}[thm]{Proposition}

\newtheorem{rem}[thm]{Remark}

\newtheorem*{claim*} {Claim}
\newtheorem*{theorem4.5'} {Theorem 4.5$'$}
\newtheorem{acknowledgment*}[thm] {Acknowledgment}
\newtheorem{example}[thm]{Example}
\newtheorem{examp}[thm]{Example}

\newtheorem*{exampleA*}{Example A}
\newtheorem*{exampleB*}{Example B}

 \newtheorem*{remark*}{Remark}
 \newtheorem{defn}[thm]{Definition}

\newtheorem{schol}[thm]{Scholium}

\newtheorem*{notation*} {Notation}
\newtheorem*{comment*} {Comment}

 \renewcommand{\sectionmark}[1]{}

\newcommand{\lm}{\lambda}

\newcommand{\om}{\omega}

 \newcommand{\id}{\operatorname{id}}

\newcommand{\SA}{\operatorname{SA}}
\newcommand{\EC}{\operatorname{EC}}
\newcommand{\AM}{\operatorname{AM}}
\newcommand{\Nac}{\operatorname{Nac}}

\newcommand{\Mod}{\operatorname{Mod}}

\begin{document}

\title[Amalgamation and extensions of summand absorbing modules]{Amalgamation and extensions \\[2mm] of summand absorbing modules  \\[2mm] over a semiring}
 \author[Z. Izhakian]{Zur Izhakian}
\address{Institute  of Mathematics,
 University of Aberdeen, AB24 3UE,
Aberdeen,  UK.}
    \email{zzur@abdn.ac.uk}
\author[M. Knebusch]{Manfred Knebusch}
\address{Department of Mathematics,
NWF-I Mathematik, Universit\"at Regensburg 93040 Regensburg,
Germany} \email{manfred.knebusch@mathematik.uni-regensburg.de}







\begin{abstract} A submodule $W$ of $V$ is
summand absorbing, if  $x + y
\in W$ implies  $x \in W, \; y \in W $ for any $x, y \in V$.  Such submodules often appear in modules over (additively) idempotent semirings, particularly in tropical algebra. This paper studies amalgamation and extensions of these submodules, and more generally of upper bound modules.   \end{abstract}

\maketitle

{ \small \tableofcontents}

\numberwithin{equation}{section}
\section*{Introduction}

This paper continues the  development of  module theory
over semirings \cite{Cos,golan92}, along the lines of classical module theory. Our approach to this theory was introduced in  \cite{Dec} and has been proceeded in \cite{SA,Gen}, starting with decompositions and generations of particular modules, termed summand absorbing modules.  The present paper focuses on amalgamations and extensions of these modules.

Semirings are extensively involved in recent studies due to increasing
 interest in tropical algebra and its applications to discrete  mathematics and automata theory.
Although our original aim was to understand modules in tropical algebra,
there are many other important examples where these modules appear, e.g.,
additive semigroups, which can be viewed as
modules over the semiring $\N_0$ of natural numbers, or sets of positive elements in an ordered
ring or a semiring.

The underlying property of these modules is
 lack of zero sums:
An $R$-module $V$ over a semiring $R$ \textbf{lacks zero sums}
(abbreviated \textbf{LZS}) or $V$ is zero-sum-free \cite[p.150]{golan92},~ if
\begin{equation}\label{eq:1.27}
  \forall \, x, y \in V : \  x + y = 0 \dss\Dir  x = y = 0. \tag{LZS}
\end{equation}
\noindent  LZS is closed for taking  submodules, direct sums, direct products, and holds for modules of
 functions $\operatorname{Fun} (S, V)$ from a set $S$ to a
module~$V$ \cite[Examples~1.6]{Dec}. For example, any  module over
an idempotent semiring is LZS \cite[Proposition~1.8]{Dec}, establishing a large assortment of
examples.

The notion of LZS leads to the next related type of submodules: A submodule $W$ is
of $V$  \textbf{summand absorbing} (abbreviated \textbf{SA}) in
$V$ (termed ``strong'' in \cite[p. 154]{golan92}), if
\begin{equation}\label{eq:1.1}
\forall \, x, y \in V: \  x + y \in W \dss\Dir x \in W, \; y \in W;
\tag{SA}
\end{equation}  $W$ is then called an \textbf{SA-submodule} of
$V$. An \textbf{SA-left ideal} of a semiring $R$ is an SA-submodule of~$R$, viewed as $R$-module by left multiplication.
An $R$-module $V$ is LZS if and only if $\{ 0_V \} $ is an SA-submodule of
$V$, thereby enhancing interest in SA-submodules.
Nevertheless, the notion of SA-submodules itself retains sense for any semiring $R$ and (left) $R$-modules $V$.

 SA-submodules arise in tropical
geometry  \cite[\S1.2]{Gen}, in supertropical algebra \cite[Example~ 1.13]{Gen}, in semigroup
theory (\cite[Note ~1.6]{Gen}). 
These submodules  have
applications to monoid semirings
\cite[Theorem~2.14]{Gen} and  to matrices over semirings  \cite[\S 2.1]{Gen}. Such matrices
have many (multiplicative) idempotents, what make them applicable to linear representations of semigroups
\cite{plc,IzMr,IzhakianMerletIdentity}.

We denote the poset (= partially ordered set) of submodules of $V$ by $\Mod(V)$, and the subposet consisting of the SA-submodules (= summand absorbing) of $V$ by $\SA(V)$. More generally, given submodules $A \supset C$ of $V$,  $\Mod(A, C)$ denotes the set of all submodules $B$ of $A$ containing $C$ and  $\SA(A, C)$ (resp. $\SA_V(A, C)$) denotes the set of SA-submodules of $A$ (resp. SA-submodules of ~$V$) containing ~$C$, i.e.,
\[ \begin{array}{rcl}
\SA (A, C) & = & \SA (A) \cap \Mod (A, C), \\[1mm]
\SA_V (A, C) & = & \SA (V) \cap \Mod (A, C).
\end{array} \]

A collection of submodules $A_1, \dots, A_n $ of $V$ has  \textbf{amalgamation} $\AM_V$, if the product $A_1 \times \cdots \times A_n $ modulo an additive exchange equivalence $\EC^V_{A_1 \times \cdots \times A_n}$ injects
in $A_1 + \cdots + A_n$ (Definition~ \ref{def:1.2}). This amalgamation induces amalgamation of SA-submodules $W_i \subset A_i$, with $W_1 + \dots + W_n \in \SA(A_1 + \dots + A_n)$ (Theorems  \ref{thm:4.5} and \ref{thm:4.6}).

We proceed in \S\ref{sec:5} with an intensive study of certain submodules and their amalgamation, involving with several supporting notions, and auxiliary results. Roughly all results pertain to families of cosets $x +D$ of a given submonoid $D$ of $V$.

A $D$-\textbf{complement} of a submodule~ $W$ in $V$ is a submodule $T$, such that $W + T = V$, $W \cap  T = D$, and $(W+ T) \cap T = \emptyset$ (Definition~ \ref{def:6.1}). When $D \in \SA(W)$,  the $D$-complement of a submodule $W$ is unique (Theorem~ \ref{thm:6.5}).

A submodule $A$ is an \textbf{SA-extension} of $D \subset A$, if $D$ is summand absorbing in $A$ (Definition~ \ref{def:7.1}). A submodule $T$ is  \textbf{complementary}
to $A$ over $D$, if $D$ is a $D$-complement of $A$ in the sum $A + T$ (Definition \ref{def:7.4}).
The SA-extension $B := [(A \sm D) + T] \cup  D$ is the \textbf{saturation} of~ $A$ by the
complementary module $T$. $A$ is $T$-\textbf{saturated}, if $A = B$. Theorem \ref{thm:7.7} links these notions:
If $A$ an SA-extension of $D$, and  $T$ is complementary
to $A$ over $D$ for which $A$ is $T$-saturated, then the pair $(A,T)$ has amalgamation
in $V$.

A submodule $T \subset V$ is \textbf{subtractive}, if for any
 $t_1, t_2 \in T$ and $x \in V$ with $x + t_1 = t_2$, also $x \in T$.
Any submodule $D \subset V$ has a unique minimal subtractive module $T$,  the \textbf{subtractive hull} of $D$. Theorem \ref{thm:9.5} lays the connections to SA-extensions:
The subtractive hull $T$ of any submodule  $D \subset V$ is
complementary over $D$ to every SA-extension of $D$.

Given a submodule $D$ of an $R$-module $V$, we have a hands a $D$-\textbf{quasiordering} on $V$defined as ($x,y \in V$)
$$ x \leq_D y \dss{\Leftrightarrow} x + d = y \text{ for some } d \in D.$$
The module $V$  is called \textbf{upper bound}, if the relation $\leq_V$ is antisymmetric and so is a (partial) ordering on $V$. (``upper bound'' refers to the fact that then $x+y$ is a built in upper bound of the set $\{x,y\}$.)

Our results pertain to upper bound $R$-modules $V$ and include the study of minimality and maximality with respect to the additive  relation $\preceq_D$ for a given submodule $D$ of $V$, defined by
$$x \preceq_D y  \dss{\Leftrightarrow} y + D \subset  x + D,$$
as well as \textbf{stable sets} $X$, i.e., sets $X$ with  $X+ D \subset ~ X$.

A special attention  is dedicated to a class of bipotent additive monoids (Definition~ \ref{def:11.2}) which can be  characterized in terms of contraction maps and convexity (Theorem~ \ref{thm:11.7}).

The paper ends with a construction of a hierarchy of families of summand absorbing submonoids
(so called \textbf{archimedean classes}, cf. Definition \ref{def:12.6}) in a suitably prepared additive monoid $V$ (Theorem \ref{thm:13.8})), which can be associated to any given additive monoid~ $V_0$ -- a  kind of ``resolution'' of $V_0$ (cf. \S\ref{sec:4}).

\section{Exchange equivalence and amalgamation}\label{sec:1}

Given a pair $(A_1, A_2)$ of submodules of $V$, we look for an equivalence relation $\sim$ on the set $A_1 \times A_2$ with the following two properties.
\begin{itemize}\dispace
\item[(1)] $(d,0) \sim (0, d)$ for all $d \in A_1\cap A_2$.
\item[(2)] The relation $\sim$ is \textbf{additive}, i.e., for pairs $(a_1, a_2)$, $(a'_1, a'_2)$, $(b_1, b_2)$, $(b'_1, b'_2)$ in $A_1 \times A_2$ with
\[ (a_1, a_2) \sim (a'_1, a'_2), \quad (b_1, b_2) \sim (b'_1, b'_2), \]
also
\[ (a_1 + b_1, a_2 + b_2) \sim (a'_1 + b'_1, a'_2 + b'_2). \]
\end{itemize}
We state an immediate consequence of these two properties.

\begin{itemize}
\item[(3)] Given $a_1, a'_1 \in A_1$, $a_2, a'_2 \in A_2$, $d \in A_1 \cap A_2$, the following holds.

\begin{itemize} \dispace
\item[3.a)] $a_1 = a'_1 + d \ds \Rightarrow (a_1, a_2) \sim (a'_1, a_2 + d)$,
\item[3.b)] $a_2 = a'_2 + d \ds \Rightarrow (a_1, a_2) \sim (a_1 + d, a'_2)$.
\end{itemize}

\end{itemize}
Property 3) leads to an explicit construction of such an equivalence relation on $A_1 \times A_2$, named ``exchange equivalence''.

We consider finite sequences of pairs $(a_1, a_2), (a'_1, a'_2), \dots, (a_1^{(k)}, a_2^{(k)})$ in $A_1 \times A_2$ with constant sum $a_1 + a_2 = a'_1 + a'_2 = \dots = a_1^{(k)} + a_2^{(k)}$. First we deal with such sequences of length $k = 2$.

\begin{defn}\label{def:1.1} Given $d \in A_1 \cap A_2$ a \textbf{(1,2)-exchange of $d$} (in $V$) is a sequence $(a_1, a_2)$, $(a'_1, a'_2)$ in $A_1 \times A_2$ with
\[ a_1 = a'_1 + d \; , \; a'_2 = a_2 + d, \]
while a (2,1)-exchange of $d$ is such a sequence with
\[ a'_1 = a_1 + d \; , \; a_2 = a'_2 + d. \]
We denote these  such exchanges symbolically by
\[ (a_1, a_2) \xrightarrow[ \quad (1,2) \quad]{d} (a'_1, a'_2), \quad (a_1', a_2') \xrightarrow[\quad(2,1)\quad]{d} (a_1, a_2), \]
respectively.
\end{defn}

We name these processes \textbf{basic exchanges}. They are very simple. In the case of a (1,2)-exchange we split off from $a_1$ a summand $d \in A_1 \cap A_2$ (in $A_1$) and add it to $a_2$ (in $A_2$).

\begin{defn}\label{def:1.2} $ $
\begin{enumerate} \ealph
  \item[a)] We call two pairs $(a_1, a_2)$, $(b_1, b_2)$ in $A_1 \times A_2$ \textbf{exchange equivalent} (in $V$) if there exists a finite sequence $(a_{10}, a_{20})$, $(a_{11}, a_{21})$, $ \dots, (a_{1k}, a_{2k})$ in $A_1 \times A_2$ starting with $(a_1, a_2) = (a_{10}, a_{20})$ and ending with $(b_1, b_2) = (a_{1k}, a_{2k})$, in which any two consecutive members $(a_{1, i - 1}, a_{2, i - 1})$, $(a_{1, i}, a_{2, i})$, are either a (1,2)-exchange or a (2,1)-exchange of some $d_i \in A_1 \cap A_2$. We then write
\[ (a_1, a_2)  \underset{A_1 \times A_2}{\sim}  (b_1, b_2). \]
and we call such sequences \textbf{chains of basic exchanges}.

  \item[b)] It is obvious that in this way we obtain an equivalence relation on the set $A_1 \times A_2$, which we name ``\textbf{exchange equivalence}'' (in $V$), and denote by $\EC_{A_1 \times A_2}$, or more elaborately by $\EC_{A_1 \times A_2}^V$.
\end{enumerate}
\end{defn}
Clearly $\EC_{A_1 \times A_2}$ has the properties (1) and (3) from above. Given two (1,2)-exchanges
\[ (a_1, a_2) \xrightarrow[\ (1,2) \ ]{d} (a'_1, a'_2) \; \mbox{ and } \; (b_1, b_2) \xrightarrow[\ (1,2) \ ]{e} (b'_1. b'_2) \]
it is plain that
\[ (a_1 + b_1, a_2 + b_2) \xrightarrow[\ (1,2) \ ]{d+e} (a'_1 + b'_1, a'_2 + b'_2). \]
The same holds for (2,1)-exchanges. It follows by an easy argument, which we defer to \S 2, that $\EC_{A_1 \times A_2}$ is additive, i.e. has the property (2) from above. Moreover $\EC_{A_1 \times A_s}$ respects scalar multiplication, i.e., for any $\lambda \in R$
\[ (a_1, a_2) \underset{A_1 \times A_2}{\sim} (a'_1, a'_2) \dss \Rightarrow (\lambda a_1, \lambda a_2) \underset{A_1 \times A_2}{\sim} (\lambda a'_1, \lambda a'_2). \]
This holds since when, say, $(a_1, a_2) \xrightarrow[\; 1,2 \;]{d} (a'_1, a'_2)$, then $(\lambda a_1, \lambda a_2) \xrightarrow[\; 1,2 \;]{\lambda d} (\lambda a'_1, \lambda a'_2)$.
Summarizing these observations we obtain

\begin{prop}\label{pro:1.3}  The exchange equivalence relation $\EC_{A_1 \times A_2}$ on the $R$-module $A_1 \times A_2$ is $R$-linear. It is the finest additive equivalence relation on $A_1 \times A_2$ with $(d, 0) \sim (0, d)$ for every $d \in A_1 \cap A_2$. \end{prop}

We denote the $\EC_{A_1 \times A_2}$-class of a pair $(a_1, a_2) \in A_1 \times A_2$ by $[a_1, a_2]_{A_1 \times A_2}$, or $[a_1, a_2]$ for short, and the set of all the classes by $A_1 \infty_V A_2$. In consequence of Proposition~1.3 we have an obvious structure of an $R$-module on $A_1 \infty_V A_2$ by defining
\begin{equation}\label{eq:1.1}
\begin{array}{rl}
 [a_1, a_2]_{A_1 \times A_2} + [b_1, b_2]_{A_1 \times A_2} & = [a_1 + b_1, a_2 + b_2]_{A_1 \times A_2} ,  \\[1mm]
 \lambda \cdot [a_1, a_2]_{A_1 \times A_2} & = [\lambda a_1, \lambda a_2]_{A_1 \times A_2},
\end{array}
 \end{equation}
where $a_1, b_1 \in A_1, a_2, b_2 \in A_2, \lambda \in R.$

\begin{defn}\label{def:1.$}  We call the $R$-module
\[ A_1 \infty_V A_2 = A_1 \times A_2/\EC_{A_1 \times A_2} \]
the \textbf{amalgamation} of $A_1$ and $A_2$ (in $V$).
\end{defn}

We furthermore have a well defined surjective $R$-module
homomorphism
$$\pi = \pi_{A_1, A_2} : A_1 \infty_V A_2 \Onto  A_1 +
A_2 \subset V$$ mapping $[a_1, a_2]$ to $a_1 + a_2$, and obtain a
natural commuting square
    \begin{equation}\label{eq:1.3}
    \begin{array}{c}
    \xymatrix{
    A_1     \ar@{->}[rr]^{j_1} & & A_1 \infty_V A_2 \\
   A_1 \cap A_2 \ar@{->}[rr]^{i_2}  \ar@{^{(}->}[u]^{i_1} &  &  A_2 \ar@{^{(}->}[u]^{j_2}
   }
    \end{array}
  \end{equation}
of $R$-module homomorphisms. Here $i_1$ and $i_2$ are the inclusion homomorphisms of $A_1 \cap A_2$ in  $A_1$ and $A_2$, and $j_1, j_2$ are given by
\begin{equation}\label{eq:1.4}
 j_1 (a_1) = [a_1, 0 ], \quad  j_2 (a_2) = [0, a_2]. \end{equation}
Clearly $\pi_1 j_1 = \id_{A_1}$, and $\pi_2 j_2 = \id_{A_2}$. Thus $A_1$ and $A_2$ embed via $j_1$ and $j_2$ into $A_1 \infty_V A_2$. In all the following we identify $A_k$ with $j_k (A_k)$ $(k = 1,2)$ and so \textit{regard} $A_1$ \textit{and} $A_2$ \textit{also as submodules of} $A_1 \infty_V A_2$. We have $A_1 + A_2 = A_1 \infty_V A_2$ and the four modules $A_1, A_2, A_1 \cap A_2$, $A_1 \infty_V A_2$ always constitute  a pushout diagram in the category of $R$-modules via there inclusion homomorphisms in $A_1 \infty_V A_2$.

\begin{defn}\label{def:1.5}  We say that a pair $(A_1, A_2)$ of submodules of $V$ \textbf{has amalgamation} (in~$V$), abbreviated $\AM_V$ or $\AM$ for short, if the map $$\pi_{A_1, A_2} : A_1 \infty_V A_2 \longrightarrow A_1 + A_2 \subset V$$ is injective, hence bijective.
\end{defn}
 This means that the four $R$-submodules $A_1, A_2$, $A_1 \cap A_2$, $A_1 + A_2$ of ~$V$ constitute a pushout diagram. In explicit terms $(A_1, A_2)$ has $\AM_V$ if for any pairs $(a_1, a_2)$, $(b_1, b_2)$ in $A_1 \times A_2$
\begin{equation}\label{eq:1.5}
 (a_1, a_2) \underset{A_1 \times A_2}{\sim} (b_1, b_2) \dss \Leftrightarrow a_1 + a_2 = b_1 + b_2. \end{equation}
If $(A'_1, A'_2)$ is a second pair of $R$-submodules of $V$ with $A'_1 \subset A_1$, $A'_2 \subset A_2$, then it is immediate from the pushout properties of $A'_1 \infty_V A'_2$ and $A_1 \infty_V A_2$ that there is a unique $R$-module homomorphism
\begin{equation}\label{eq:1.6}
 \kappa_{A'_1 \times A'_2, A_1 \times A_2} : A'_1 \infty_V A'_2 \longrightarrow A_1 \infty_V A_2 \end{equation}
with
\begin{equation}\label{eq:1.7}
 \pi_{A_1, A_2} \circ \kappa_{A'_1 \times A'_2, A_1 \times A_2} = \pi_{A'_1, A'_2}. \end{equation}
It sends an element $[a_1, a_2]_{A'_1 \times A'_2}$ to $[a_1, a_2]_{A_1 \times A_2}$. In other terms, for any $(a_1, a_2)$, $(b_1, b_2)$ in $A_1 \times A_2$
\begin{equation}\label{eq:1.8}
 (a_1, a_2) \underset{A'_1 \times A_2'}{\sim} (b_1, b_2) \dss\Rightarrow (a_1, a_2) \underset{A_1 \times A_2}{\sim} (b_1, b_2). \end{equation}
This is also immediate from our explicit description of the exchange equivalence relation.

The map (1.6) is injective iff the restriction $\EC_{A_1 \times A_2} \vert A'_1 \times A'_2$ of $\EC_{A_1 \times A_2}$ to the subset $A'_1 \times A'_2$ of $A_1 \times A_2$ coincides with $\EC_{A'_1 \times A'_2}$. We then write
\[ A'_1 \infty_V A'_2 \subset A_1 \infty_V A_2, \]
regarding $A'_1 \infty_V A'_2$ as a submodule of $A_1 \infty_V A_2$.

\begin{prop}\label{prop:1.6.} When $(A'_1, A'_2)$ and $(A_1, A_2)$ are pairs of $R$-submodules of $V$
with $A'_1 \subset A_1$, $A'_2 \subset A_2$, and $(A_1, A_2)$ has
$\AM_V$, then $(A'_1, A'_2)$ also has $\AM_V$ iff $E_{A_1 \times
A_2} \vert A'_1 \times A'_2$ coincides with $E_{A'_1 \times A'_2}$.
\end{prop}

\begin{proof} We have a commuting square
    \begin{equation*}\label{eq:1.3}
     \begin{array}{c}
    \xymatrix{
A'_1 \infty_V A'_2  \ar@{->}[d]^{\pi'}     \ar@{->}[rr]^{\kappa} & & A_1 \infty_V A_2 \ar@{->}[d]^{\pi} \\
   A'_1 + A'_2 \ar@{^{(}->}[rr]   &  &  A_1 + A_2
   }\end{array}
  \end{equation*}
with $\pi' = \pi_{A'_1, A'_2}$, $\pi = \pi_{A_1, A_2}$, $\kappa = \kappa_{A'_1 \times A'_2, A_1 \times A_2}$, and the inclusion map of the submodule $A'_1 + A'_2$ of $V$ into $A_1 \times A_2$. By assumption $\pi$ is injective. Then $\pi'$ is injective iff $\kappa$ is injective. \end{proof}

We state some obvious facts about exchange equivalence.

\begin{rem}\label{rem:1.7} Assume that $(A_1, A_2)$ is a pair of submodules of $V$ and $a_1, b_1 \in A_1$, $a_2, b_2 \in A_2$.
\begin{itemize}\dispace
\item[i)] $(a_1, a_2) \underset{A_1 \times A_2}{\sim} (b_1, b_2) \dss \Leftrightarrow (a_2, a_1) \underset{A_2 \times A_1}{\sim} (b_2, b_1)$.\\
Thus $A_1 \infty_V A_2 = A_2 \infty_V A_1$. The pair $(A_1, A_2)$ has $\AM_V$ iff $(A_2, A_1)$ has $\AM_V$.
\item[ii)] If $A_1 \supset A_2$, then $(a_1, a_2) \underset{A_1 \times A_2}{\sim} (a_1 + a_2, 0)$, and so $(A_1, A_2)$ has $\AM_V$.
\item[iii)] Let $\varphi : V \to V'$ be an $R$-module homomorphism. Then
\[ (a_1, a_2) \underset{A_1 \times A_2}{\sim} (b_1, b_2) \dss\Rightarrow (\varphi (a_1), \varphi (a_2)) \underset{\varphi (A_1) \times \varphi (A_2)}{\sim} (\varphi (b_1), \varphi (b_2)). \]
\end{itemize}
\end{rem}

\section{Additivity of exchange equivalences}\label{sec:2}

We verify the additivity of $\EC_{A_1 \times A_2}$ asserted in \S 1 by using ``normalized'' chains of exchanges.

If a chain of basic exchanges in $A_1 \times A_2$ is given, connecting $(a_1, a_2)$ to $(b_1, b_2)$, we can simplify this chain in various ways. First note that if, say
\[ (a_1, a_2) \xrightarrow[\quad 1,2 \quad]{d} (a'_1, a'_2), \quad (a'_1, a'_2) \xrightarrow[\quad 1,2 \quad]{e} (a''_1, a''_2), \]
then $(a_1, a_2) \xrightarrow[\ 1,2 \ ]{d+e} (a''_1, a''_2)$. Thus we can achieve that in the chain from $(a_1, a_2)$ to $(b_1, b_2)$ the basic exchanges alternate between type $(1,2)$ and $(2,1)$. We call such a chain a \textbf{zig-zag}.

Furthermore, we have the trivial basic exchange $(a_1, a_2) \xrightarrow[\ 1,2 \ ]{0} (a_1, a_2)$ which coincides with $(a_1, a_2) \xrightarrow[\ 2,1 \ ]{0} (a_1, a_2)$. By employing trivial basic exchanges, we can achieve that the zig-zag from $(a_1, a_2)$ to $(b_1, b_2)$ also starts with a basic exchange of type $(1,2)$ and ends with one of type $(2,1)$. We call such a chain of exchanges a \textbf{normalized zig-zag} (an admittedly ad hoc notion). Finally we can increase the length of the normalized zig-zag by 1, adding a trivial basic exchange, or by 2, adding a chain $(a_1, a_2) \xrightarrow[\ 1,2 \ ]{d} (a'_1, a'_2) \xrightarrow[\ 2,1 \ ]{d} (a_1, a_2)$ of length 2 with $d \in A_1 \cap A_2$.

\begin{prop}\label{prop:2.1} Assume that $(a_1, a_2)$, $(b_1, b_2)$, $(u_1, u_2)$, $(w_1, w_2)$ are pairs in $A_1 \times A_2$ with
\[ (a_1, a_2) \underset{A_1 \times A_2}{\sim} (b_1, b_2), \quad (u_1, u_2) \underset{A_1 \times A_2}{\sim} (w_1, w_2).\]
Then $(a_1 + u_1, a_2 + u_2) \underset{A_1 \times A_2}{\sim} (b_1 + w_1, b_2 + w_2)$.
\end{prop}
\begin{proof} We can choose two normalized zig-zags in $A_1 \times A_2$ of same length connecting $(a_1, a_2)$ to $(b_1, b_2)$ and $(u_1, u_2)$ to $(w_1, w_2)$. By adding these zig-zags in the obvious way we obtain a normalized zig-zag in $A_1 \times A_2$ connecting $(a_1 + u_1, a_2 + u_2)$ to $(b_1 + w_1, b_2 + w_2)$.
\end{proof}

\section{Pairs of SA-submodules with amalgamation}\label{sec:3}

Recall that, assuming $A$ is a submodule of $V$,  a submodule $W$ of $A$ is SA in $A$ if for any two elements $u, v \in A$ with $u + v \in W$, also $u \in W$ and $v \in W$. More generally, then a finite sum $a_1 + \dots + a_n$ of elements of $A$ is in $W$ iff every $a_i \in W$.
We are ready for a central result of this paper.

\begin{thm}\label{thm:3.1} Assume that $(A_1, A_2)$,\ $(W_1, W_2)$ are pairs of submodules of $V$ with $W_1 \subset A_1$, $W_2 \subset A_2$ and
\[ W_1 \cap A_2 \subset W_2, \quad A_1 \cap W_2 \subset W_1, \]
and hence
\[ W_1 \cap A_2 = A_1 \cap W_2 = W_1 \cap W_2. \]
Assume furthermore that $(A_1, A_2)$ has amalgamation in $V$ and
that $W_1 \in \SA (A_1)$, $W_2 \in \SA (A_2)$. Then $(W_1, W_2)$ has
amalgamation in $V$.
\end{thm}
\begin{proof} We will  verify that $\EC_{A_1 \times A_2} \vert W_1 \times W_2 = \EC_{W_1 \times W_2}$. Then, we know by Proposition~1.6 that $(W_1, W_2)$ has $\AM_V$. Due to our explicit description of exchange equivalence in \S 1 it suffices to verify the
following:

Given $a_1, b_1 \in W_1$, $a_2, b_2 \in W_2$, $d \in A_1 \cap A_2$ with either
\[ (a_1, a_2) \xrightarrow[\quad 1,2 \quad]{d} (b_1, b_2) \quad \mbox{or} \quad (a_1, a_2) \xrightarrow[\quad 2,1 \quad]{d} (b_1, b_2) \]
in $A_1 \times A_2$, then these moves are exchanges in $W_1 \times W_2$.
In the first case we have
\[ a_1 = b_1 + d \; , \quad b_2 = a_2 + d. \]
From $a_1 \in W_1$ we conclude that $b_1, d \in W_1$. It follows that $d \in W_1 \cap A_2 = W_1 \cap W_2$. Thus $(a_1, a_2) \xrightarrow[\; 1,2 \;]{d} (b_1, b_2)$ is a (1,2)-exchange in $W_1 \times W_2$. The second case is settled in the same way. \end{proof}

The question arises whether $W_1 + W_2$ is again SA in $A_1 + A_2$.

\begin{thm}\label{thm:3.2} Under the assumptions of Theorem \ref{thm:3.1}  the following also holds.
\begin{itemize}\dispace
\item[a)] If $(a_1, a_2) \in A_1 \times A_2$, $(w_1, w_2) \in W_1 \times W_2$ and $(a_1, a_2) \underset{A_1 \times A_2}{\sim} (w_1, w_2)$, then $(a_1, a_2) \in W_1 \times W_2$. In other words, the subset $W_1 \times W_2$ of $A_1 \times A_2$ is a union of $\EC_{A_1 \times A_2}$-equivalence classes.
\item[b)] $W_1 + W_2 \in \SA (A_1 + A_2)$.
\end{itemize}
\end{thm}
\begin{proof} a): It suffices to verify that $(a_1, a_2) \in W_1 \times W_2$ in the special cases that
\[ (a_1, a_2) \xrightarrow[\quad 1,2 \quad]{d} (w_1, w_2) \quad \mbox{or} \quad (a_1, a_2) \xrightarrow[\quad 2,1 \quad]{d} (w_1, w_2) \]
due to our description of $\EC_{A_1 \times A_2}$ in \S 1. In the first case we have $w_2 = a_2 + d$, $a_1 = w_1 + d$, and $d \in A_1 \cap A_2$. Since $W_2$ is SA in $A_2$, we conclude that $a_2 \in W_2$ and $d \in W_2 \cap (A_1 \cap A_2) = A_1 \cap W_2 = W_1 \cap W_2$, whence $a_1 = w_1 + d \in W_1 + W_1 \cap W_2 = W_1$, as desired. In the second case $w_1 = a_1 + d$, $a_2 = w_2 + d$, and again $d \in A_1 \cap A_2$. Since $W_1 \in \SA (A_1)$, we have $a_1 \in W_1$, $d \in W_1 \cap (A_1 \cap A_2) = W_1 \cap A_2 = W_1 \cap W_2$, and so $a_2 = w_2 + d \in W_2 + W_1 \cap W_2 = W_2$, as desired.

\pSkip
b): Given $u, v \in A_1 + A_2$ with $u + v \in W_1 + W_2$, we need to verify that $u, v \in W_1 + W_2$. We write $u = a_1 + a_2$, $v = b_1 + b_2$, $w = w_1 + w_2$ with $a_1, b_1 \in A_1$, $a_2, b_2 \in A_2$, $w_1 \in W_1$, $w_2 \in W_2$. Now
\[ (a_1 + b_1) + (a_2 + b_2) = w_1 + w_2. \]
Since $(A_1, A_2)$ has $\AM_V$, this implies
\[ (a_1 + b_1, a_2 + b_2) \underset{A_1 \times A_2}{\sim} (w_1, w_2), \]
and thus, as proved above,
\[ a_1 + b_1 \in W_1 \; , \quad a_2 + b_2 \in W_2. \]
Since $W_1 \in \SA (A_1)$, $W_2 \in \SA (A_2)$, we conclude that $a_1, b_1 \in W_1$ and $a_2, b_2 \in W_2$. Thus $u = a_1 + a_2 \in W_1 + W_2$ and $v = b_1 + b_2 \in W_1 + W_2$. \end{proof}

\begin{cor}\label{cor:3.3} Assume that $(A_1, A_2)$ has $\AM_V$ and that $A_1 \cap A_2 \in \SA (A_2)$. Then $A_1 \in \SA (A_1 + A_2)$.
\end{cor}
\begin{proof} Apply Theorem 3.2.b with $W_1 = A_1$, $W_2 = A_1 \cap A_2$.
\end{proof}

\section{Multiple amalgamation}\label{sec:4}

We expand the exchange equivalence relation on $A_1 \times A_2$ for pairs $(A_1, A_2)$ in $\Mod (V)$ to an ``exchange equivalence`` on $A_1 \times \dots \times A_n$ for $n$-tuples $(A_1, \dots, A_n)$ in $\Mod (V)$, $n \geq 3$.

\begin{defn}\label{def:4.1} Let $(a_1, \dots, a_n)$ and $(b_1, \dots, b_n)$ be tuples in $A_1 \times \dots \times A_n$.
\begin{itemize}
\item[a)] We say that $n$-tuples $(a_1, \dots, a_n)$ and $(b_1, \dots, b_n)$ are \textbf{binary exchange equivalent} (in $A_1 \times \dots \times A_n)$, if there exist $i, j \in \{ 1, \dots, n \}$, $i \ne j$, such that $a_k = b_k$ for $k \ne i, j$ and $(a_i, a_j)  \underset{A_i \times A_j}{\sim} (b_i, b_j)$, (as defined in \S\ref{sec:1}). More specifically we then say that these tuples are $(i, j)$-(exchange)-equivalent, and write
\[ (a_1, \dots, a_n) \overset{i,j}{\sim} (b_1, \dots, b_n), \]
or more elaborately
\[ (a_1, \dots, a_n) \raisebox{-6pt}{$\overset{i,j}{\scriptstyle \widetilde{A_1 \times \dots \times A_n}}$}  (b_1, \dots, b_n). \]
\item[b)] We say that $(a_1, \dots, a_n)$ and $(b_1. \dots, b_n)$ are \textbf{exchange equivalent} (in $A_1 \times \dots \times A_n)$, and write
\[ (a_1, \dots, a_n) \underset{A_1 \times \dots \times A_n}{\sim} (b_1, \dots, b_n), \]
if there is a finite chain of tuples in $A_1 \times \dots \times A_n$ starting with $(a_1, \dots, a_n)$ and ending with $(b_1, \dots, b_n)$, such that any two consecutive members are binary exchange equivalent.

This is clearly an equivalence relation on the set $A_1 \times \dots \times A_n$. We denote it by
$\EC_{A_1 \times \dots \times A_n}^{V}$ or $\EC_{A_1 \times \dots \times A_n}$ for short.
\end{itemize}
\end{defn}

\begin{prop}\label{prop:4.2} $\EC_{A_1 \times \dots \times A_n}$ is additive. In other words, when tuples $(a_1, \dots, a_n)$, $(b_1, \dots, b_n)$, $(u_1, \dots, u_n)$, $(w_1, \dots, w_n)$ in $A_1 \times \dots \times A_n$ are given with
\[ (a_1, \dots, a_n) \underset{A_1 \times \dots \times A_n}{\sim} (b_1, \dots, b_n) \dss{\text{ and }}  (u_1, \dots, u_n) \underset{A_1 \times \dots \times A_n}{\sim} (w_1, \dots, w_n), \]
then
\[ (a_1 + u_1, \dots, a_n + u_n) \underset{A_1 \times \dots \times A_n}{\sim} (b_1 + w_1, \dots, b_n + w_n). \]
\end{prop}
\begin{proof} We pick chains of binary equivalences in $A_1 \times \dots \times A_n$ from $(a_1, \dots, a_n)$ to $(b_1, \dots, b_n)$ and from $(u_1, \dots, u_n)$ to $(w_1, \dots, w_n)$. For any $(x_1, \dots, x_n) \in A_1 \times \dots \times A_n$ we have the trivial equivalence $(x_1, \dots, x_n) \sim (x_1, \dots, x_n)$, which is a binary $(i, j)$-equivalence for any two different $i, j \in \{ 1, \dots, n \}$. Inserting trivial equivalences in both chains, we refine them to chains of some length with $(i, j)$-exchanges at the same places. Adding the refined chains in $A_1 \times \dots \times A_n$, we obtain a chain of binary equivalences from $(a_1 + u_1, \dots, a_n + u_n)$ to $(b_1 + w_1, \dots, b_n + w_n)$, as desired. \end{proof}

It is evident from the case $n = 2$ that $(a_1, \dots, a_n) \overset{i, j}{\sim} (b_1, \dots, b_n)$ implies $(\lambda a_1, \dots, \lambda a_n) \overset{i, j}{\sim} (\lambda b_1, \dots, \lambda b_n)$ for any $\lambda \in R$. We conclude that $\EC_{A_1 \times \dots \times A_n}$ is an $R$-linear equivalence relation on the $R$-module $A_1 \times \dots \times A_n$. The following is now obvious, as in the case $n = 2$ (Proposition~1.3), and tells us that we have been on the right track.

\begin{thm}\label{thm:4.3} $\EC_{A_1 \times \dots \times A_n}$ is the finest additive equivalence relation on $A_1 \times \dots \times A_n$ with the property that for any (different) $i, j \in \{ 1, \dots, n \}$ and $d \in A_i \cap A_j$
\[ (0, \dots, 0, \underset{i}{d}, 0, \dots, 0) \ds \sim (0, \dots, 0, \underset{j}{d}, 0, \dots, 0). \]
This relation is $R$-linear. \end{thm}

We denote the $\EC_{A_1 \times \dots \times A_n}$-equivalence class
of a tuple $(a_1, \dots, a_n) \in A_1 \times \dots \times A_n$ by
$[a_1, \dots, a_n]_{A_1 \times \dots \times A_n}$, or $[a_1, \dots,
a_n]$ for short. The rules for addition and multiplication by
scalars $\lambda \in R$ for $n = 2$ (cf. \S\ref{sec:1}) generalize
in the obvious way, and establish the structure of an $R$-module on
the set $A_1 \times \dots \times A_n/\EC_{A_1 \times \dots \times
A_n}^{\vee}$ of equivalence classes, which we call again the
\textbf{amalgamation} of $(A_1, \dots, A_n)$ (in $V$). We denote
this $R$-module by $A_1  \infty_V \cdots \infty_V A_n$. As in the
case $n = 2$ we have a well defined $R$-linear surjection
\begin{equation}\label{eq:4.1}
  \kappa_{A_1 \times \cdots \times A_n} : A_1 \infty   \cdots \infty A_n \dss \Onto A_1 + \dots + A_n \subset V
\end{equation}
mapping $[a_1, \dots, a_n]$ to $a_1 + \cdots + a_n$ \footnote{When
the ambient module $V$ of the $A_i$ is fixed, we often omit the
suffix ``$V$''.} and $R$-linear injections
\[ j_k : A_k \longrightarrow A_1 \infty \cdots \infty A_n \]
mapping $a \in A_k$ to $[0, \dots, 0, a, 0, \dots, 0]$, with entry $a$ at the place $k$. It allows us to identify~ $A_k$ with the submodule $j_k (A_k)$ of $A_1 \infty \cdots \infty A_n$ whenever this is appropriate. Note that
\begin{equation}\label{eq:4.2} A_1 \infty \cdots \infty A_n = \sum\limits_{k = 1}^n j_k (A_k). \end{equation}

\begin{defn}\label{def:4.4} As in the case $n = 2$, we say that the
 sequence $(A_1, \dots, A_n)$ has \textbf{amalgamation} in $V$, abbreviated $\AM_V$, if $\kappa_{A_1, \dots, A_n}$ is injective,
  and so is an $R$-module isomorphism from $A_1 \infty_V \cdots \infty_V A_n$ to $A_1 + \dots + A_n \subset V$. \end{defn}

Intuitively we regard the map $$\kappa_{A_1, \dots, A_n} : A_1 \infty \dots \infty A_n \Onto A_1 + \dots + A_n$$ as a kind of ``resolution'' of the submodule $V : = A_1 + \dots + A_n$ of $V$ with respect to the family $(A_1, \dots, A_n)$ of modules generating $V$.

\begin{thm}\label{thm:4.5} Assume that $(A_1, \dots, A_n)$ and $(W_1, \dots, W_n)$ are $n$-tuples of submodules of ~$V$ with $W_k \in \SA (A_k)$ for $k = 1, \dots, n$ and $W_i \cap A_j \subset W_j$, whence $W_i \cap A_j = W_i \cap W_j$ for any two (different) $i, j \in \{ 1, \dots, n \}$. Then, if $(A_1, \dots, A_n)$ has $\AM_V$, also  $(W_1, \dots, W_n)$ has ~$\AM_V$.
\end{thm}

\begin{proof} Follows from the Definition \ref{def:4.1} of exchange equivalence and Theorem~\ref{thm:3.1}.
\end{proof}
\begin{thm}\label{thm:4.6} Under the assumption of Theorem \ref{thm:4.5} on $(A_1, \dots, A_n)$ and $(W_1, \dots, W_n)$ the following holds.
\begin{itemize}\dispace
\item[a)] If the  tuples of vectors $(a_1, \dots, a_n) \in A_1 \times \dots \times A_n$, $(w_1, \dots, w_n) \in W_1 \times \dots \times W_n$,  are exchange equivalent in $A_1 \times \dots \times A_n$, then $(a_1, \dots, a_n) \in W_1 \times \dots \times W_n$, and both tuples are exchange equivalence in $W_1 \times \dots \times W_n$.
\item[b)] $W_1 + \dots + W_n \in \SA (A_1 + \dots + A_n)$.
\end{itemize}
\end{thm}
\begin{proof} a): We know by Theorem \ref{thm:3.2}.a that,  if $(a_1, \dots, a_n) \overset{i, j}{\sim} (w_1, \dots, w_n)$ in $A_1 \times \dots \times A_n$, then $(a_1, \dots, a_n) \in A_1 \times \dots \times A_n$ and both tuples are $(i, j)$-equivalent in $W_1 \times \dots \times W_n$. Thus any chain of binary exchanges in  $A_1 \times \dots \times A_n$, which meets $W_1 \times \dots \times W_n$, is a chain of binary exchanges in $W_1 \times \dots \times W_n$.
\pSkip
b): The argument used in the proof of Theorem \ref{thm:3.2}.b works also here with the obvious changes. \end{proof}

We clarify the contents of Theorems \ref{thm:4.5} and \ref{thm:4.6}
in a special case.

\begin{cor}\label{cor:4.7}
 Assume that $A_1, \dots, A_n$ are submodules of $V$ and that $W$ is an SA-submodule of $A_1 + \dots + A_n$.
\begin{itemize} \dispace
\item[a)] If $(A_1, \dots, A_n)$ has amalgamation in $V$, then $(W \cap A_1, \dots, W \cap A_n)$ has  amalgamation in $V$.
\item[b)] Two $n$-tuples $(a_1, \dots, a_n)$, $(b_1, \dots, b_n)$ of $(W \cap A_1) \times \dots \times (W \cap A_n)$ are exchange equivalent in $A_1 \times \dots \times A_n$ iff these tuples are exchange equivalent in $(W \cap A_1) \times \dots \times (W \cap A_n)$.
\item[c)] $(W \cap A_1) \times \dots \times (W \cap A_n)$ is a union of $E_{A_1 \times \dots \times A_n}$ equivalence classes of $A_1 \times \dots \times A_n$.
\end{itemize}
\end{cor}
\begin{proof} Apply Theorems \ref{thm:4.5} and \ref{thm:4.6} to $(A_1, \dots, A_n)$ and $W_i : = W \cap A_i$ $(1 \leq i \leq n)$.
\end{proof}

For later use we  restate Theorem \ref{thm:4.3}  more categorically.
This is possible, since dividing out an $R$-linear equivalence
relation on an $R$-module $X$ means to establish a surjective
homomorphism $\varphi : X \twoheadrightarrow Z$, which is unique up
to composition with an isomorphism $\delta: Z \isoto
 Z'$ replacing $\varphi$ by $\delta
\circ \varphi$.

\begin{defn}\label{def:4.8} Let $A_1, \dots, A_n \in \Mod (V)$. The \textbf{amalgamation diagram} of $(A_1, \dots, A_n)$ (in $V$) is the diagram of $R$-module homomorphisms consisting of all inclusion maps $A_i \cap A_j \hookrightarrow  A_i$ for different $i, j \in [1, n ]$ (in $V$). \end{defn}

\begin{schol}\label{schol:4.9} A tuple
$(A_1, \dots, A_n)$ has amalgamation in $V$ iff the amalgamation
diagram of $(A_1, \dots, A_n)$ in $V$ admits the submodule $A_1 +
\dots + A_n$ of $V$ as colimit with the inclusion maps $A_i
\hookrightarrow A_1 + \dots + A_n$ $(1 \leq i \leq n)$ as morphisms.
More explicitly, given $R$-linear maps $\alpha_i : A_i
\rightarrow Z$ $(1 \leq i \leq n)$ into some $R$-module $Z$ with
$\alpha_i \vert A_i \cap A_j = \alpha_j \vert A_i \cap A_j$ for $i
\ne j$, there exists a (unique) $R$-linear map $\gamma : A_1 + \dots
+ A_n \rightarrow Z$ such that $\gamma \vert A_i = \alpha_i$ for
$1 \leq i \leq n$. \end{schol}

\section{Formal properties of multiple exchange equivalence and amalgamation}\label{sec:5}

The following fact is trivial but useful.

\begin{prop}\label{prop:5.1}  Let  $A_1, \dots, A_n \in \Mod (V)$, $n \geq 2$, be given and let $\sigma$ be a permutation of $[1, n] : = \{ 1, \dots, n \}$.
\begin{itemize}
\item[a)] If $(a_1, \dots, a_n)$, $(b_1, \dots, b_n)$ are $n$-tuples in $A_1 \times \dots \times A_n$, then \\
$(a_1, \dots, a_n) \underset{A_1 \times \dots \times A_n}{\sim} (b_1, \dots, b_n)\quad \text{iff}\quad
(a_{\sigma (1)}, \dots, a_{\sigma (n)}) \underset{A_{\sigma (1)}
\times \dots \times A_{\sigma (n)}}{\sim} (b_{\sigma (1)}, \dots,
b_{\sigma (n)}).$
\item[b)] The tuple $(A_1, \dots, A_n)$ has amalgamation in $V$ iff $A_{\sigma (1)} \times \dots \times A_{\sigma (n)}$ has amalgamation in $V$.
\end{itemize}
\end{prop}

\begin{proof} Evident, since the amalgamation diagram of $(A_{\sigma (1)},  \dots , A_{\sigma (n)})$ arises from the amalgamation diagram of $(A_1, \dots, A_n)$ simply by relabeling the $A_i$ using $\sigma$.
\end{proof}

We briefly discuss the functorial behaviour of exchange equivalence and amalgamation under a module homomorphism.

\begin{prop}\label{prop:5.2}  Assume that $\varphi$ is a linear map from an $R$-module $V'$ to $V$, and that $A_1, \dots, A_n$ are submodules of $V$.
\begin{itemize}
\item[a)] If $(a'_1, \dots, a'_n)$ and $(b'_1, \dots, b'_n)$ are tuples in $\varphi^{-1} (A_1) \times \dots \times \varphi^{-1} (A_n)$ with
\begin{equation}\label{eq:5.1}
 (a'_1, \dots, a'_n) \underset{\varphi^{-1} (A_1) \times \dots \times \varphi^{-1} (A_n)}{\sim} (b'_1, \dots , b'_n), \end{equation}
then
\begin{equation}\label{eq:5.2}
 (\varphi (a'_1), \dots, \varphi (a'_n)) \underset{A_1 \times \dots \times A_n}{\sim} (\varphi (b'_1), \dots, \varphi (b'_n)). \end{equation}

\item[b)] Assume that $\varphi$ is injective on $\varphi^{-1} (A_1 + \dots + A_n)$, then \eqref{eq:5.2} implies~\eqref{eq:5.1}. Then  $(\varphi^{-1} (A_1), \dots , \varphi^{-1} (A_n))$ has amalgamation in $V'$ iff $(A_1, \dots, A_n)$ has amalgamation in $V$.
\end{itemize}
\end{prop}

\begin{proof} a): Evident from the description of the exchange equivalence relation in Definitions~ \ref{def:1.2} and \ref{def:4.1}.\pSkip
b): Without loss of generality we may replace the $R$-modules $V$
and $V'$ by $A_1 + \dots + A_n$ and $\varphi^{-1} (A_1 + \dots +
A_n)$. Now $\varphi$ is an isomorphism from $V'$ onto $V$ and
$\varphi^{-1} (A_1 + \dots + A_n) = \varphi^{-1} (A_1) + \dots +
\varphi^{-1} (A_n)$, whence $\varphi$ provides an isomorphism
between the amalgamation diagrams of $(\varphi^{-1} (A_1), \dots,
\varphi^{-1} (A_n))$ in $V'$ and $(A_1, \dots, A_n)$ in $V$. Due to
step a), applied to $\varphi$ and $\varphi^{-1}$, it is now obvious
that for fixed tuples $(a'_1, \dots, a'_n)$ and $(b'_1, \dots ,
b'_n)$ in $\varphi^{-1} (A_1) \times \dots \times \varphi^{-1}
(A_n)$ the assertions \eqref{eq:5.1} and \eqref{eq:5.2} are
equivalent. We furthermore have a commuting square
    \begin{equation*}\label{eq:5.3}
    \xymatrix{
\varphi^{-1} (A_1) \infty \cdots \infty \varphi^{-1} (A_n)  \ar@{->}[d]^{\varphi_1 \infty \cdots \infty \varphi_n}     \ar@{->}[rr]^{\qquad \kappa'} & & V' \ar@{->}[d]^{\vrp} \\
   A_1 \infty \cdots \infty A_n \ar@{->}[rr]   &  &  V
   }
  \end{equation*}
with $\kappa':= \kappa_{\varphi^{-1} (A_1) \times \cdots \times \varphi^{-1} (A_n)}$, $\kappa : = \kappa_{A_1} \times \cdots \times \kappa_{A_n}$ (cf. \eqref{eq:4.2}) and $\varphi_1 \infty \cdots \infty \varphi_n$ the isomorphism from $\varphi^{-1} (A_1) \infty \cdots \infty \varphi^{-1} (A_n)$ to $A_1 \infty \cdots \infty A_n$ induced by the isomorphisms $\varphi_i : \varphi^{-1} (A_i) \isoto A_i$ obtained from $\varphi$ by restriction. Since the vertical arrows in this diagram are isomorphisms, the map $\kappa'$ is an isomorphism iff $\kappa$ is an isomorphism. This means that $(\varphi^{-1} (A_1), \dots, \varphi^{-1} (A_n))$ has amalgamation in $V'$ iff $(A_1, \dots, A_n)$ has amalgamation in $V$. \end{proof}

We now delve deeper into the theory of amalgamation.

\begin{thm}\label{thm:5.3} Let $(A_1, \dots, A_n)$ be an $n$-tuple of submodules of $V$ and $A_0$ a submodule of ~$A_1$.
\begin{itemize}\dispace
\item[a)] If $(A_1, \dots, A_n)$ has $\AM_V$, then $(A_0, A_1, \dots, A_n)$ has $\AM_V$.
\item[b)] Given tuples $(a_0, a_1, \dots, a_n)$, $(b_0, b_1, \dots, b_n)$ in $A_0 \times \dots \times A_n$ we have
    $$
    \begin{array}{cl}
          (a_0, a_1, \dots, a_n) \underset{A_0 \times \dots \times A_n}{\sim} (b_0, b_1, \dots , b_n) & \Leftrightarrow \\ (a_0 + a_1, a_2, \dots, a_n) \underset{A_1 \times \dots \times A_n}{\sim} (b_0 + b_1, b_2, \dots, b_n). &
       \end{array}
 $$
\end{itemize}
\end{thm}
\begin{proof} a): Suppose $R$-module homomorphisms $\alpha_i : A_i \rightarrow Z$ $(0 \leq i \leq n)$ are given with $\alpha_i \vert A_i \cap A_j = \alpha_j \vert A_i \cap A_j$ $(0 \leq i, j \leq n)$. Then (exploiting this for $1 \leq i, j \leq n)$ we have a (unique) $R$-module homomorphism $\gamma : A_1 + \dots + A_n \rightarrow Z$ with $\gamma \vert A_i = \alpha_i$ for $1 \leq i \leq n$. Since $A_0 \cap A_1 = A_0$ we have $\alpha_0 = \alpha_1 \vert A_0$. Since $A_0 \cap A_i \subset A_i$ for $i \geq 2$, we have $\alpha_0 \vert A_0 \cap A_i = \alpha_1 \vert A_0 \cap A_i = \alpha_i \vert A_0 \cap A_i$. Thus $A_0 + A_1 + \dots + A_n = A_1 + \dots + A_n$ is the colimit of the amalgamation diagram of $(A_0, A_1, \dots, A_n)$ in $V$ (with the obvious inclusion morphisms).
\pSkip b): $(A_1, \dots, A_n)$ has $\AM$ in $A_1 \infty \cdots
\infty A_n$,  so $(A_0, A_1, \dots, A_n)$ also has  $\AM$ in $A_1
\infty \cdots \infty A_n$, and $A_0 \infty A_1 \infty \cdots \infty
A_n = A_1 \infty \cdots \infty A_n$. Now b) is obvious, since the
inclusion maps of the $A_i$ into $A_1 \infty \cdots \infty A_n = A_0
\infty \cdots \infty A_n$ give us $E_{A_1 \times \dots \times A_n}$
and $E_{A_0 \times A_1 \times \dots \times A_n}$, and $(a_0, a_1,
\dots, a_n)$, $(b_0, b_1, \dots, b_n)$ have the same images as $(a_0
+ a_1, a_2, \dots, a_n)$, $(b_0 + b_1, b_2, \dots, b_n)$
respectively in $A_1 \infty \cdots \infty A_n$.
\pSkip
In short, b) follows from a) by working in $A_1 \infty \cdots \infty A_n$ instead of $A_1 + \dots + A_n$.
\end{proof}

Given a sequence $A_1, \dots, A_n$ of submodules of $V$, as before,  we define for any set $J \subset [1, n]$
\begin{equation}\label{eq:5.4}
A_J : = \sum\limits_{i \in J} A_i . \end{equation}

\begin{thm}[Contraction Theorem]\label{thm:5.4} Assume that $(A_1, \dots, A_n)$ has amalgamation in $V$, and that
\begin{equation}\label{eq:5.5}
 [1, n] = J_1 \ \dot{\cup} \ds\cdots \dot{\cup} \ J_m \end{equation}
is a partition of $[1, n]$ into $m$ (disjoint) subsets. Then $(A_{J_1}, \dots, A_{J_m})$ has amalgamation~ in ~$V$.
\end{thm}
\begin{proof}  In view of Proposition \ref{eq:5.1},
we assume without loss of generality that the $J_i$ are consecutive subintervals $[1, r_1], [r_1 + 1, r_2], \dots, [r_{m-1} + 1, r_m]$ with $1 < r_1 < r_2 < \dots < r_m =~ n$. Then we infer by an easy induction argument, that it suffices to consider the subcase $m = n-1$, $J_1 = [1, r]$ for some $r$ with $1 < r <n$ and $\vert J_t \vert = 1$ for $r < t \leq n$. In other words, we only need to verify that $(A_1 + \dots + A_r, A_{r+1}, \dots, A_n)$ has $\AM_V$.

Looking at the appropriate amalgamation diagrams we see that our
task is as follows. Given $R$-linear maps $\beta : A_1 + \dots + A_r
\rightarrow Z$, $\alpha_i : A_i \rightarrow Z$ for $r < i \leq n$,
with
\[ \beta \vert (A_1 + \dots + A_r) \cap A_i = \alpha_i \vert (A_1 + \dots + A_i) \cap A_j \]
for $r < i \leq n$ and
\[ \alpha_i \vert A_i \cap A_j = \alpha_j \vert A_i \cap A_j \]
for $r < i < j \leq n$, we have to find an $R$-linear map $\gamma : A_1 + \dots + A_n \rightarrow Z$ with $\gamma \vert A_1 + \dots + A_r =~ \beta$ and $\gamma \vert A_i = \alpha_i$ for $ r + 1 \leq i \leq n$.

This is easy. We define $\alpha_i : = \beta \vert A_i$ for $1 \leq i \leq r$. Then clearly $\alpha_i \vert A_i \cap A_j = \alpha_j \vert A_i \cap A_j$ for $1 \leq i < j \leq r$ and also for $i \in [1, r]$, $j > r$. Since $(A_1, \dots, A_n)$ has $\AM_V$, there is an $R$-linear map $\gamma : A_1 + \dots + A_r \to Z$ with $\gamma \vert A_i = \alpha_i$ for all $i \in [1, n]$. Due to the additivity of $\gamma$ it follows that $\gamma \vert A_1 + \dots + A_r = \beta \vert A_1 + \dots + A_r$ and $\gamma \vert A_i = \alpha_i$ for $r < i \leq n$. \end{proof}

We point out a consequence of Theorems \ref{thm:5.3} and \ref{thm:5.4}.

\begin{cor}\label{cor:5.5}
 Given submodules $A_0, A_1, \dots, A_n$ of $V$, assume that $(A_0, A_1, \dots, A_n)$ has $\AM_V$. Then $(A_0 + A_1, A_0 + A_2, \dots, A_0 + A_n)$ has $\AM_V$.
\end{cor}
\begin{proof} By iterated application of Theorem \ref{thm:5.3} (and using  the trivial Proposition~\ref{prop:5.1}) we infer that the 2n-tuple $(A_0, A_1, A_0, A_2, \dots, A_0, A_n)$ has $\AM_V$. It follows by the Contraction Theorem~\ref{thm:5.4}, that $(A_0 + A_1, A_0 + A_2, \dots, A_0 + A_n)$ has $\AM_V$. \end{proof}

\section{$D$-complements}\label{sec:6}

In the papers \cite{Dec} and \cite{SA} our main focus  has been on
the lattice $\SA (V)$ of all $\SA$-submodules of a module $V$ over a
semiring $R$. As a counterpart to this  we
study  for \textbf{any}  $\SA$-submodule $D$ of $V$ in \S\ref{sec:7} below the set of
submodules $A \supset D$ of $V$. Here the amalgamation theory will
come into play.

In preparation for this endeavor we present general results about
submodules of $V$.

\begin{defn}\label{def:6.1} Let $D \subset W$ be submodules of an $R$-module $V$. We call a submodule $T$ of ~ $V$ a $D$-\textbf{complement of} $W$ \textbf{in} $V$, if $W + T = V$, $W \cap T = D$, and $(w + T) \cap T = \emptyset$ for every $w \in W \setminus D$. \end{defn}

In the case of $D = \{ 0 \}$ these are the \textit{weak complements}
studied in \cite{Dec}. We  generalize part of the theory of weak
complements in \cite{Dec} to $D$-complements.

\begin{prop}\label{prop:6.2}  Assume that $T$ is an $\SA$-submodule of $V$. Then every submodule $W$ of ~$V$ with $W + T = V$ has $T$ as a $D$-complement in $V$ for $D : = W \cap T$. (N.B.: $D \in \SA (W)$).
\end{prop}

\begin{proof}  Let  $w \in W \sm D$ and $t \in T$. Then $w + t \notin T$ since $w+ t \in T$ would imply that  $w \in W \cap T = D$. This proves that $(w + T) \cap T = \emptyset$.  \end{proof}

In the following, through Theorem \ref{thm:6.5}, $D$ and $W$ are
submodules of $V$ with $D \subset W$.

\begin{prop}\label{prop:6.3}  Assume that $T$ is a $D$-complement of $W$ in $V$, and that $D \in \SA (W)$. Then $T \in \SA (V)$. \end{prop}

\begin{proof} Let $v_1, v_2 \in V$ with $v_1 + v_2 \in T$. Write $v_1 = w_1 + t_1$, $v_2 = w_2 + t_2$ ($w_i \in W$, $t_i \in T$). Then
\[ (w_1 + w_2) + (t_1 + t_2) \in T. \]
This forces
$w_1 + w_2 \in  D$, since otherwise
$ (w_1 + w_2 +T) \cap T = \emptyset$. We conclude that $w_1 \in D \subset T$, and so $w_i + t_i \in T$ for $i =1,2$. \end{proof}

\begin{prop}\label{prop:6.4}  Assume that $D \in \SA (W)$,  that $T$ is a $D$-complement of $W$ in $V$, and that $U$ is a submodule of $V$ with $W + U = V$. Then $T \subset W$.
\end{prop}

\begin{proof} Let $t \in T$. Write $t = w + u$ with $w \in W$, $u \in U$. We infer from Proposition \ref{prop:6.3} that $T \in \SA (V)$, and conclude that $w \in T \cap W = D \subset U$. Thus $t \in U$.
\end{proof}

\begin{thm}\label{thm:6.5} Assume that $T$ and $U$ are $D$-complements of $W$ in $V$ and that $D \in \SA (W)$. Then $T = U$.
\end{thm}
\begin{proof}  By Proposition \ref{prop:6.4}, $T \subset U$ and $U \subset T$. \end{proof}
Note that up to this point no amalgamation hypothesis has been used.

\begin{rem}\label{rem:6.6}
  If $(A_1, A_2)$ is a pair of submodules of $V$ with amalgamation in $V$, where $D : = A_1 \cap A_2$ is $\SA$ in $A_1$, then Corollary~\ref{cor:3.3} states that $A_2$ is a $D$-complement of $A_1$ in $A_1 + A_2$, the unique one by Theorem~\ref{thm:6.5},
  and  moreover
    $A_2$ is $\SA$ in $A_1 + A_2$ (which is also clear by Proposition~\ref{prop:6.4}).
     If also $D \in \SA (A_1)$, it follows that $D \in \SA (A_1 + A_2)$, and that ~ $A_1$ is a $D$-complement of $A_2$ in $A_1 + A_2$, which
     again is unique.

$$
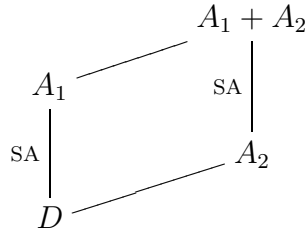
\xymatrixrowsep{3mm}
\xymatrixcolsep{6mm}
    \xymatrix{
    & & A_1 + A_2 \ar@{-}[dd]_\SA  \\
    A_1      \ar@{-}[rru] \ar@{-}[dd]_\SA & & \\
    & & A_2  \\
    D \ar@{-}[rru] & & \\
   }
$$
\end{rem}

\section{SA-extensions and their saturations by complementary submodules}\label{sec:7}

As before, $V$ is a module over a semiring $R$.

\begin{defn}\label{def:7.1} An $\SA$-\textbf{extension in} $V$ is a pair of submodules $(A, D)$ of $V$ where $D \subset A$ and $D$ is $\SA$ in $A$. We then also say that $A$ is an $\SA$-\textbf{extension of} $D$.
\end{defn}

We state some  easy facts about these extensions.

\begin{rem}\label{rem:4.9} $ $
\begin{enumerate}\dispace
  \item [a)] Given submodules $D \subset A \subset B$ of $V$, the following holds: If $(B, D)$ is an SA-extension, then $(A, D)$ is an SA-extension. If both $(A, D)$ and $(B, A)$ are SA-extensions, then $(B, D)$ is an SA-extension.

\item[b)] Assume that $D$ is a submodule of $V$ and $(A_{\alpha} \vert \alpha \in I)$ is a family of SA-extensions of $D$ such that for any $\alpha, \beta \in I$ there exists $\gamma \in I$ with $A_{\alpha} \subset A_{\gamma}$, $A_{\beta} \subset A_{\gamma}$. Then $\bigcup\limits_{\alpha \in I} A_{\alpha}$ is an SA-extension of $D$.

\item[c)] It follows by Zorn's Lemma that for any given SA-extension $A$ of $D$ there exists a maximal SA-extension $C$ of $D$ with $A \subset C$.

\end{enumerate}

\end{rem}

We characterize SA-extensions in the following intrinsic way.

\begin{prop}\label{prop:7.3}  Let $D \subset A$ be submodules of $A$. Then $A$ is an SA-extension of $D$ iff $A \setminus D$ is closed under addition and $(A \setminus D) + D \subset A \setminus D$. In this case $(A \setminus D) + D = A \setminus D$. \end{prop}

\begin{proof}  $D$ is SA in $A$ iff
\[\forall \ a_1, a_2 \in A: \, a_1 + a_2 \in D
\dss\Rightarrow a_1 \in D, a_2 \in D. \]
Since both subsets $A$ and $D$ of $V$ are closed under addition, this condition can be rewritten as
\[ (A \setminus D) + (A \setminus D) \subset A \setminus D, \qquad  (A \setminus D) + D \subset A \setminus D. \]
Furthermore $(A \setminus D) + D = A \setminus D$, since $0 \in D$. \end{proof}

We slightly extend the definition of a $D$-complement in $V$ (Definition \ref{def:6.1}) as follows.

\begin{defn}\label{def:7.4} Given submodules $D \subset A$ in $V$, we call a submodule $T$ of $V$ \textbf{complementary to $A$ over $D$}, if $D$ is a $D$-complement of $A$ in $A + T$, i.e.,
\begin{equation}\label{eq:7.1}
A \cap T = D, \qquad [(A \setminus D) + T] \cap T = \emptyset. \end{equation}
\end{defn}

This situation will be studied in this and the following sections in
the case that $A$ is an SA-extension of $D$. By Remark
\ref{rem:6.6}, $A + T$ is an SA-extension of $T$, and $T$ is
uniquely determined by the submodules $A$ and $A + T$ of $V$.

We can enlarge an SA-extension without changing a given complementary module as follows.

\begin{thm}\label{thm:7.5} Assume that $D \subset T$ are submodules of $V$ and that $A$ is an SA-extension of~ $D$ with complementary module $T$. Then $B : = [(A \setminus D) + T] \cup D$ is also such an SA-extension, and $B + T = A + T$.\end{thm}

\begin{proof} We have $A \cap T = D$ and $[(A \setminus D) + T] \cap T = \emptyset$. Since $D \subset T$, this implies $[(A \setminus D) + T] \cap D = \emptyset$, whence $B \setminus D = (A \setminus D) + T$. This set is closed under addition, since $A \setminus D$ is closed under addition. Furthermore
\[ \begin{array}{rcl}
B + T & = & [(A \setminus D) + T] \cup (D + T)  \\[1mm]
      & = & [(A \setminus D) + T] \cup T = A + T, \\
\end{array} \]
and
\[ (B \setminus D) + D = (A \setminus D) + T + D =(A \sm D) + T=B \setminus D. \]
We verify that for given $z \in B$, $\lambda \in R$, also $\lambda z \in B$.

If $z \in D$, then $\lambda z \in D \subset B$. Otherwise $z = x + t$ with $x \in A \setminus D$, $t \in T$. Now $\lambda z = \lambda x + \lambda t$. We have $\lambda x \in A \subset B$ and $\lambda t \in T \subset B$, whence again $\lambda z \in B$. Thus $B$ is an $R$-submodule of $V$. Since $B \setminus D$ is closed under addition and $(B \setminus D) + D = B \setminus D$, we conclude by Proposition~\ref{prop:7.3} that $B$ is an SA-extension of $D$. We have $[(B \setminus D) + T] \cap T = [(A \setminus D) + T] \cap T = \emptyset$, and so $T$ is complementary to $B$ over $D$. \end{proof}

Note that $B$ is the smallest submonoid of $(V, +)$ containing $A$, such that $B \setminus D$ is a union of cosets $x + T$ of $T$ in $V$.

\begin{defn}\label{def:7.6}$ $

\begin{enumerate}\dispace
  \item[  a)] We call the SA-extension $B : = [(A \setminus D) + T] \cup D$ the \textbf{saturation} of $A$ by the complementary module $T$.

\item[b)] If $B = A$, i.e., $(A \setminus D) + T = A \setminus D$, we say that $A$ is $T$-\textbf{saturated} (or \textbf{saturated w.r.t.} $T$).
\end{enumerate}
 \end{defn}

Our interest in saturated SA-extensions is due to the following fact.

\begin{thm}\label{thm:7.7} Assume that $(A, D)$ is an SA-extension in $V$, and that $T$ is a complementary module to $A$ over $D$, for which $A$ is $T$-saturated. Then the pair $(A, T)$ has amalgamation in~ $V$.
\end{thm}
\begin{proof} Given additive maps $\alpha : A \to Z$, $\beta : T \to Z$ into an $R$-module $Z$ with $\alpha \vert D = \beta \vert D$ we need to establish an additive map $\gamma$ from $A + T = (A \setminus D)\, \dot{\cup}\, T$ to $Z$ with $\gamma \vert A = \alpha$, $\gamma \vert T = \beta$. We define a map $\gamma$ from the set $(A \setminus D)\, \dot{\cup} \, T$ to $Z$ by the rule
\begin{equation}\label{eq:7.2}
 \begin{array}{lll}
& \gamma (a) = \alpha (a) \;& \mbox{for} \; a \in A \setminus D, \\[2mm]
& \gamma (t) = \beta (t) \; &\mbox{for} \; t \in T. \end{array} \end{equation}
Then $\gamma \vert T = \beta$, $\gamma \vert A \setminus D = \alpha \vert A \setminus D$, $\gamma \vert D = \beta \vert D = \alpha \vert D$. Given $x_1 = a_1 + t_1$, $x_2 = a_2 + t_2$ with $a_1, a_2 \in A \setminus D$, $t_1, t_2 \in T$,we have $x_1 + x_2 = (a_1 + a_2) + (t_1 + t_2)$. Since $A \setminus D$ is closed under addition, the summand $a_1 + a_2$ is in $A \setminus D$, while $t_1 + t_2$ is in $T$. Thus
\[ \begin{array}{rcl}
\gamma (x_1 + x_2) & = & \alpha (a_1 + a_2) + \beta (t_1 + t_2)  \\[1mm]
                   & = & \alpha (a_1) + \alpha (a_2) + \beta (t_1) + \beta (t_2) \\[1mm]
                   & = & [ \alpha (a_1) + \beta (t_1)] + [ \alpha (a_2) + \beta (t_2) ]\\[1mm]
                   & = & \gamma (x_1) + \gamma (x_2). \\
\end{array} \]
If $x_1 = a_1 + t_1$ and $x_2 = t_2$ with $a_1 \in A \setminus D$,
$t_1 \in T$, $t_2 \in T$, then $x_1 + x_2 = a_1 + (t_1 + t_2)$, and
so
\[ \begin{array}{rcl}
\gamma (x_1 + x_2) & = & \alpha (a_1) + \beta (t_1 + t_2) \\[1mm]
                   & = & \alpha (a_1) + \beta (t_1) + \beta (t_2) \\[1mm]
                   & = & \gamma (x_1) + \gamma (x_2) \\
\end{array} \]
again. If $x_1 , x_2 \in T$, then $\gamma (x_1 + x_2) = \gamma (x_1) + \gamma (x_2)$, since $\gamma \vert T = \beta$ is additive. Thus $\gamma$ is indeed additive. \end{proof}

\begin{rem}\label{rem:7.8}
 If $(A, D)$ is an SA-extension in $V$ and $T$ is complementary to $A$ over $D$, then any submodule $T'$ of $V$
 with $D \subset T' \subset T$ is again complementary to $A$ over $D$, since $[(A \setminus D) + T] \cap T = \emptyset$
 implies $[(A \setminus D) + T'] \cap T' = \emptyset$. If moreover $(A, D)$ is $T$-saturated, then $(A, D)$ is $T'$-saturated, since
 $(A \setminus D) + T = A \setminus D$ implies $(A \setminus D) + T' = A \setminus D$. Thus $(A, T')$ has $\AM_V$. \end{rem}

The latter fact is noteworthy, since in general, when $(A_1, A_2)$ is pair of submodules of $V$ with $\AM_V$, and $A'_2$ is a submodule of $A_2$ containing $A_1 \cap A_2$, the pair $(A_1, A'_2)$ can fail to have $\AM_V$.

In good cases it is possible to desend the results above for $\SA$-extensions of $D$ in $V$ to $\SA$-extensions of $\{ 0_V \}$ in $V$.

\begin{thm}\label{thm:7.9} Assume that every element of the semiring $R$ is a sum of units of $R$ (e.g. $R = \mathbb{N}_0$). Assume furthermore that $(A, D)$ is an $\SA$-extension in $V$ and $T$ is a complementary module of $A$ over $D$. Then the subset
\[ A_0 : = (A \setminus D) \cup \{ 0 \} \]
of $V$ is again a submodule of $V$ and $(A_0, \{ 0 \})$ is an
$\SA$-extension of $\{ 0 \}$ in $V$ with complementary module $T$.
The $\SA$-extension $(A_0, \{ 0 \})$ is $T$-saturated iff $(A, D)$
is $T$-saturated. Then the pair $(A_0, T)$ has $\AM_V$.
\end{thm}
\begin{proof} We have $(A_0 \setminus \{ 0 \}) = A \setminus D$, and so the set $A_0 \setminus \{ 0 \}$ is closed under addition in $V$. Let $\lambda \in R$ be given. Then $\lambda = \sum\limits_{i \in I} \lambda_i$ with finitely many $\lambda_i \in R^*$. (We admit $I = \emptyset$. Then read $\lambda = 0$.) For every $i \in I$ the map $x \mapsto \lambda_i x$ is an automorphism of the monoid $(A, +)$ which restricts to an automorphism of $(D, +)$. Thus $\lambda_i (A \setminus D) = A \setminus D$. We conclude that $\lambda_i A_0 = A_0$ for every $i \in I$, whence $\lambda A_0 \subset A_0$. Thus $A_0$ is an $R$-submodule of $V$.

Since $T$ is complementary to $A$ over $D$, i.e., $A \cap T = D, \; [ (A \setminus D) + T] \cap T = \emptyset$,
we have
\begin{equation}\label{eq:str}
(A_0 \setminus \{ 0 \}) + T = (A \setminus D) + T, \tag{$*$}
\end{equation} and so $(A_0 \setminus \{ 0 \} + T) \cap T =
\emptyset$. Furthermore
$$\begin{array}{llllll}
A_0 + T &= [(A_0 \setminus \{ 0 \}) \cup \{ 0 \}] + T
& = [(A_0 \setminus \{ 0 \}) + T] \cup ( \{ 0 \} + T) \\[1mm] && = [ (A \setminus D) + T] \cup T \quad = A + T,
\end{array}
$$ and
$$\begin{array}{lll}
 A_0 \cap T  &=  [(A_0 \setminus \{ 0 \}) \cap T] \cup ( \{ 0 \} \cap T)
           =  [ (A \setminus D) \cap T] \cup \{ 0 \} = \{ 0 \}.
\end{array}$$
We conclude from ($\ast$) that $A_0$ is $T$-saturated iff $A$ is
$T$-saturated. Then $(A_0, T)$ has $\AM_V$ by Theorem~\ref{thm:7.7}.
\end{proof}

\section{The complementary modules of a fixed SA-extension}\label{sec:8}

We want to get a hold on the set of submodules of $V$ which are complementary to a given SA-extension in $V$. In the beginning we work without an SA-assumption.

\begin{prop}\label{prop:8.1}  Let $D \subset A$ be submodules of $V$, and assume that $T$ is a complementary module to $A$ over $D$ in $V$. Further assume that $U'$ is a submodule of $U : = A + T$ containing ~$D$. Then
\begin{equation}\label{eq:8.1}
 T': = \{ x \in T \ds \vert A + Rx \subset U' \} \end{equation}
is a submodule of $V$ which is again complementary to $A$ over $D$, and $A + T' = U'$.
\end{prop}
\begin{proof}  If $x_1, x_2 \in V$ and $A + R x_1 \subset U'$, $A + R x_2 \subset U'$, then for any $\lambda_1, \lambda_2 \in R$
\[ A + R (\lambda_1 x_1 + \lambda_2 x_2) \subset A + R x_1 + R x_2 \subset U', \]
and so $\lambda_1 x_1 + \lambda_2 x_2 \in T'$. Thus $T'$ is an $R$-submodule of $V$ with $T' \subset T$. Clearly $A + T' \subset U'$. Given $x \in U'$, we have $A + R x \subset U'$, and so $x \in T'$. Thus $U' \subset A + T'$. This proves $A + T' = U'$. Since $D \subset U'$, also $D \subset T'$, and so $D \subset A \cap T' \subset A \cap T = D$, which proves that $A \cap T' = D$. We have $(A \setminus D) + T' \subset (A \setminus D) + T$, and this is disjoint from $T$, all the more from $T'$. Thus $T'$ is complementary to $A$ over $D$. \end{proof}

\begin{rem}\label{rem:8.2} $ $
\begin{itemize} \dispace
  \item[a)] It is now obvious from \eqref{eq:7.2} that $T'$ is the largest submodule $F$ of $T$ with $A + F \subset U'$.
      \item[b)] In the case $R = \mathbb{N}_0$ we have the following alternative description of $T'$.
\begin{equation}\label{eq:8.2}
 T' = \{ x \in T \ds \vert A + x \subset U' \}. \end{equation}
To verify this, let $\Phi$ denote the set on the right hand side. Clearly $0 \in \Phi$. If $x_1, x_2$ are elements of $T$ with $A + x_1 \subset U'$, $A + x_2 \subset U'$, then $A + x_1 + x_2 = A + A + x_1 + x_2 \subset U' + U' = U'$. Thus $\Phi$ is an $\mathbb{N}_0$-submodule of $T$ with $A + \Phi \subset U'$. It is clear from~ \eqref{eq:8.1} that $T' \subset \Phi$. We conclude by a) that $T' = \Phi$.
\end{itemize}
\end{rem}

Proposition \ref{prop:8.1} leads us to the following picture.
Given submodules $D \subset A$ of $V$ we define two sets of submodules of $V$.
\begin{align*}
  \Compl'_D (A)& : = \{ T \in \mbox{Mod} (V, D)\ds  \vert A \cap T = D, \ [(A \setminus T) + T] \cap T = \emptyset \}  \\[1mm]
\Compl''_D (A)& : = \{ U \in \mbox{Mod} (V, A) \ds \vert A \; \mbox{has a} \; D\mbox{-complement in} \; U \}
\end{align*}
We regard these sets as subposets of Mod$(V, D)$ and Mod$(V, A)$ respectively. It is evident that $\Compl'_D (A)$ is a lower set in Mod$(V, D)$, and it follows from Proposition~\ref{prop:8.1} that $\Compl''_D (A)$ is a lower set in Mod$(V, A)$. We have a surjective map $T \mapsto A + T$ from $\Compl'_D (A)$ to $\Compl''_D (A)$, which respects the partial orderings of these sets.

Assume now that $(A, D)$ is an SA-extension, i.e. $D \in \SA (A)$. Then this map is bijective, since the $D$-complement $T$ of $A$ in $U = A + T$ is uniquely determined by $A$ and $U$. Thus we have an isomorphism of posets
\begin{equation*}\label{eq:8.5}
 \Compl'_D (A)  \xrightarrow[+A]{\quad \sim \quad } \Compl''_D (A). \end{equation*}
$\Compl'_D (A)$ and $\Compl''_D (A)$ have the bottom elements $D$ and $D + A = A$ respectively. Note also that the union of the modules in a chain in $\Compl'_D (A)$ is again an element of $\Compl'_D (A)$, and so by Zorn's Lemma every element of $\Compl'_D (A)$ is contained in a maximal element of $\Compl'_D (A)$. These are the maximal modules complementary to $A$ over $D$.

\section{SA-extensions with a fixed complementary module}\label{sec:9}

We now fix submodules $D \subset T$ in $V$ and search for the SA-extensions $A$ of $D$, for which~ $T$ is a complementary module over $D$, i.e. (cf. Definition \ref{def:7.4})
\[ A \cap T = D, \qquad  [ (A \setminus D) + T ] \cap T = \emptyset. \]

\begin{rem}\label{rem:9.1} $ $
\begin{enumerate}\dispace
  \item[ a)] If $(A, D)$ is such an SA-extension then every subextension $(A_1, D)$, $D \subset A_1 \subset A$ is again an SA-extension with complementary module $T$.

\item[b)] If $(A_i \ds \vert i \in I)$ is a chain of SA-extensions of $D$ with complementary module $T$, then the union $\bigcup\limits_{i \in I} A_i$ is again such an SA-extension. Thus every SA-extension of $D$ with complementary module $T$ is contained in a maximal such extension. Of course such an extension is $T$-saturated.

\end{enumerate}
     \end{rem}

Our next goal is, to exhibit a module $T \supset D$ which serves as a complementary module for \textbf{every} SA-extension of $D$.

\begin{defn}[{\cite[p.154]{golan92}}]\label{def:9.2}  A submodule $T$ of $V$ is \textbf{subtractive} (in $V$), if for any two elements $t_1, t_2 \in T$ and $x \in V$ with $x + t_1 = t_2$, also $x \in T$, in other terms
\[ \forall \ x \in V: \; (x + T) \cap T \ne \emptyset \dss\Rightarrow x \in T. \]
\end{defn}

It is evident that the intersection of any family of subtractive submodules of $V$ is again subtractive. Thus for a submodule $D$ of $V$ there is a unique minimal subtractive module $T \supset D$ in $V$, namely the intersection of all subtractive modules containing $D$. We call this module $T$ the \textbf{subtractive hull} of $D$ (Golan uses the term ``subtractive closure'' [loc.cit., p. 155]).

We look for a more explicit description of the subtractive hull of a given submodule~ $D$ of ~$V$.

\begin{defn}\label{def:9.3}  We say that a vector $x \in V$ \textbf{has no $D$-access to} $D$ if $(x + D) \cap D = \emptyset$, and  denote the set of all these vectors by $\Nac_D (D)$. Thus $\Nac_D (D)$ is the union of all cosets $x + D$ in $V$ which are disjoint from $D$.
\end{defn}

In the following we denote the set $\Nac_D (D)$ briefly by $N$ and its complement $V \setminus N$ in $V$ by $N^c$. Clearly $N^c$ is the set of all $x \in V$ with $(x + D) \cap D \ne \emptyset$, i.e., the set of all $x \in V$ such that there exist $d, d' \in D$ with $x + d = d'$.

\begin{prop}[Michihiro Takahashi, cf. {\cite[p.155]{golan92}}]\label{prop:9.4} $N^c = \{ x \in V \ds \vert (x + D) \cap D \ne \emptyset \}$ is the subtractive hull of $D$ in $V$. \end{prop}

\begin{proof} a) We verify that $N^c$ is a submodule of $V$. Given $x_1, x_2 \in N^c$, we have elements $d_1, d'_1, d_2, d'_2$ of $D$ with $x_1 + d_1 = d'_1$, $x_2 + d_2 = d'_2$, and so $(x_1 + x_2) + (d_1 + d_2) = d'_1 + d'_2$, whence $x_1 + x_2 \in N^c$. Given $x \in N^c$ and $\lambda \in R$ we have some $d, d' \in D$ within $x + d = d'$. It follows that $\lambda x + \lambda d = \lambda d'$. Thus $\lambda x \in N^c$.
\pSkip
b) We verify that the module $N^c$ is subtractive in $V$. Let $x_1, x_2 \in N^c$, $v \in V$, and $x_1 + v = x_2$. There are vectors $d_1, d'_1$, $d_2, d'_2 \in D$ with $x_1 + d_1 = d'_1$, $x_2 + d_2 = d'_2$. Adding $d_1$ to the equation $x_1 + v = x_2$ we obtain
\[ d'_1 + v = x_2 + d_1. \]
Then adding $d_2$, we obtain
\[ d'_1 + v + d_2 = d_1 + d'_2. \]
This proves that $v \in N^c$.
\pSkip
c) Since for any $x \in N^c$ there are elements $d, d'$ of $D$ with $x + d = d'$, it is obvious that $x \in T$ for any subtractive module $T \supset D$, whence $T \supset N^c$. \end{proof}

The papers of M. Takahashi cited in \cite[p.155]{golan92} have been inaccessible for us. Thus we felt obliged to give a detailed proof of Proposition~\ref{prop:9.4}. Golan [loc.cit.] uses the notation $E^V_D (D)$ for the subtractive hull of $D$.

\begin{thm}\label{thm:9.5} Let $D$ be any submodule of $V$. Then the subtractive hull $T$ of $D$ in $V$ is complementary over $D$ to every SA-extension of $D$.
\end{thm}
\begin{proof} We write again $N : = \Nac_D (D)$. Then $N = V \setminus T$ by Proposition~9.4. Let $A$ be an SA-extension of $D$. Since $D$ is SA in $A$, we have $[(A \setminus D) + D] \cap D = \emptyset$, i.e., $A \setminus D \subset N$. This means that $(A \setminus D) \cap T = \emptyset$. It follows that
\[ A \cap T = [(A \setminus D) \cap T] \cup (D \cap T) = D. \]
This proves that $T$ is complementary to $A$ over $D$ (cf. \eqref{eq:7.1}). \end{proof}

It is natural to ask for a given semiring $R$ and modules $D \subset V$, whether the subtractive hull of $D$ is the \textit{maximal} submodule $T \supset D$ which is complementary over $D$ for every SA-extension of $D$ in $V$. We exhibit a situation where this is true.

\begin{prop}\label{prop:9.6} Assume that $R$ is a zerosumfree semifield. Assume furthermore that $T$ is the subtractive hull of $D$ in $V$ and $T'$ is a submodule of $V$, which properly contains $T$. Let $x \in T' \setminus T$. Then the module $A : = D + Rx$ is an SA-extension of $D$, for which $T'$ is \textbf{not} complementary to $A$ over $D$ (while of course $T$ is complementary to $A$ over $D$ by Theorem~\ref{thm:9.5}).
\end{prop}

\begin{proof} It suffices to verify that $D$ is SA in $A$, which by Proposition~\ref{prop:7.3} means, that $(A \setminus D) + D$ is disjoint from $D$, and $A \setminus D$ is closed under addition. Then $T'$ is certainly not complementary to $A$ over $D$ since $A \subset T'$.

Let $z = \lambda x + d \in A \setminus D$ with $\lambda \in R$, $d \in D$. Then $\lambda \in R \setminus \{ 0 \}$, and so $\lambda \in R^*$. Suppose that $z + d_1 = d'_1$ for some $d_1, d'_1 \in D$. Then $\lambda x + d + d_1 = d'_1$. Since $D \subset T$ and $T$ is subtractive, this implies $\lambda x \in T$ and then $x \in T$, a contradiction. Thus $[(A \setminus D) + D] \cap D = \emptyset$.

If $z_1 = \lambda_1 x + d_1$, $z_2 = \lambda_2 x + d_2$ are elements of $A \setminus D$, then
\[ z_1 + z_2 = (\lambda_1 + \lambda_2) x + d_1 + d_2. \]
We have $\lambda_1 \ne 0$, $\lambda_2 \ne 0$, and so $\lambda_1 + \lambda_2 \ne 0$, whence $z_1 + z_2 \in A \setminus D$. Thus $A \setminus D$ is closed under addition. \end{proof}

Similar results about SA-extensions of $D$ can be obtained by starting with the set
\[ N_0 : = \Nac_V (D) := \{ x \in V \ds\vert (x + V) \cap D = \emptyset \} \]
of vectors in $V$ ``without $V$-access to $D$'', instead of $N = \Nac_D (D)$. Then
\[ N_0^c = \{ x \in V \ds\vert \exists\ v \in V: \, x + v \in D \} \]
is again an $R$-submodule of $V$ containing $D$. It is the downset $D^{\downarrow}$ of the set $D$ with respect to the minimal preordering $\preceq_V$ of $V$ discussed in \cite[\S 6]{Dec} (in equivalent terms, the convex hull of $D$ in this preordering.) Since the convex submodules w.r. to $\preceq_V$ are precisely the SA-submodules of $V$ \cite[Prop. 6.7]{Dec} it is clear that $N_0^c$ \textbf{is the} SA-\textbf{closure of} $D$, i.e., the smallest SA-submodule $T$ of $D$ which contains $D$. This SA-closure appears in \cite[p.155]{golan92} (citing M. Takahashi) under the name \textit{strong closure} of $D$, notated there by $E_V^V (D)$. Since $N_0^c = D^{\downarrow}$,
\begin{equation}\label{eq:9.1}
 N_0 = \{ x \in V \ds \vert x \not\in D^{\downarrow} \}, \end{equation}
and we conclude easily (cf. \cite[p.155]{golan92}) that
\begin{equation}\label{eq:9.2}
 N_0 + D^{\downarrow} = N_0 + V = N_0. \end{equation}
Of course, $N_0 \subset N$, more precisely
\[ N_0 = \{ x \in N \ds \vert (x + V) \cap D = \emptyset \} = \{ x \in N \ds \vert x + V \subset N \}. \]
In close analogy to the arguments in the proofs of Proposition~\ref{prop:9.4} and Theorem~\ref{thm:9.5}, and using Theorem~\ref{thm:7.5}, we obtain

\begin{thm}\label{thm:9.7} $A_0 : = N_0 \ds{\dot{\cup}} D$ is an SA-extension of $D$ with $A_0 \cap D^{\downarrow} = D$, and $D^{\downarrow}$ is the (unique) $D$-complement of $A_0$ in $N_0 + D^{\downarrow} = V$. 
\end{thm}

It is now plain that $A_0$ is the unique maximal SA-extension $C$ of $D$ with $C \setminus D \subset N_0$. Thus the SA-hull $D^{\downarrow}$ is the $D$-complement in $V$ of just one maximal SA-extension, namely $N_0 \cup D$, while usually the subtractive hull $N^c$ of $D$ shows up as the $D$-complement of several maximal SA-extensions $A$ of $D$ in $A + N^c$.

\section{Minimal $D$-cosets}\label{sec:10}

Given a submodule $D$ of an $R$-module $V$, $R$  any semiring, we introduce a  binary relation~ $\leq_D$ on $V$ as follows:
\begin{equation}\label{eq:10.1}
  x \leq_D y \dss\Leftrightarrow \exists d \in D: x+ d = y.
\end{equation}
This relation is a quasiordering, i.e., it is transitive and reflexive, but not necessarily antisymmetric. Obviously $\leq_D$ is compatible with scalar multiplication:   If $x \leq_D y$, then  $\lm x \leq_D \lm y$ for any $\lm \in R$.
We call $\leq_D$ the \textbf{$D$-quasiordering} on $V$.

Below we restrict to the case that $\leq_D$ is antisymmetric, and so is a partial ordering on $V$, named the \textbf{$D$-ordering} on $V$. We then also say that $V$ is $D$-ordered. The case of $D =V$ has attracted interest for long. Such an $R$-module $V$ is called \textbf{upper bound} (abbreviated u.b.), since then the sum $x+y$ of two  vectors $x,y$ is an upper bound of the set $\{ x,y \}$. It is well known that $V$ is u.b. iff the $R$-module $V$ lacks zero sums \eqref{eq:1.27}, i.e., is zero-sum-free in the terminology of \cite[p. 156]{golan92}. (A long list of such modules is given in \cite[\S1]{SA}.)
An u.b. $R$-module is obviously $D$-ordered for any submodule $D$ of $V$. Thus our focus on $D$-ordered modules is a  rather mild restriction.

We turn to $D$-cosets $x + D$ in a $D$-ordered $R$-module $V$.
\begin{rem}\label{rem:10.1} Let $x,y \in D$. Then, obviously,
$$ y+ D \subset x +D \dss \Leftrightarrow y \in x+ D  \dss \Leftrightarrow x \leq_D y.$$
Thus $x+ D = y+ D \ds \Leftrightarrow x \leq_D y$ and $y \leq_D x \ds \Leftrightarrow x =y$, since
$\leq_D$ is antisymmetric.
  \end{rem}
\begin{defn}\label{def:10.2} We define on a $D$-ordered $R$-module $V$ the binary relation $\preceq_D$ as
$$ x \preceq_D y \dss \Leftrightarrow y +D \subset x +D.$$
\end{defn} \noindent
This relation is again a partial ordering on the set $V$. But, without more assumptions on $R$, there is no reason that $\preceq_D$ is compatible with scalar multiplication.

We search for minimal $D$-cosets, i.e., maximal vectors with respect to $\preceq_D$, which show up in connection with a pair $(A,T)$ of submodules of $V$ with amalgamation, where $A \cap T = D$, $A$ is an SA-extension of $D$ in $V$, and $T$ is a $D$-complement of $A$ in $V$. More specifically we pick an SA-extension $D \subset A $ in the $D$-ordered $R$-module $V$. Thus we know that $A \sm D$ is closed under addition , and that
$(A \sm D) + D = A\sm D$. We define
\begin{equation}\label{eq:10.2}
\Ddw := \{ x\in V \ds | \exists d \in D: x \leq_D d\}.
\end{equation}
Thus $\Ddw$ is the set of all $x \in V$ with $x +d = d'$ for some $d, d' \in D$. In other words, $\Ddw$ is the subtractive hull of $D$ in $V$. We verify directly (without involving \S\ref{sec:9}) that
\begin{equation}\label{eq:10.3}
  [(A \sm D) + \Ddw] \cap \Ddw = \emptyset.
\end{equation}
Indeed, suppose there exist $x \in A \sm D$, $y \in \Ddw$ with $x + y \in \Ddw$. Then
$y + d_1 = d_2$, $x+y + d_3 = d_4$ for some $d_i \in D.$ This implies
$x+ y+ d_3 + d_4 = d_4 +d_1$, and then $x+d_2 +d_3 = d_1 +d_4$, in contradiction to
$(A \sm D) + \Ddw \subset A \sm D$.
Thus $\Ddw$ is complementary to $A$ over $D$ (Definition \ref{def:7.4}) and we infer from Theorem \ref{thm:7.5} that
\begin{equation}\label{eq:10.4}
  B : = [(A \sm D) + \Ddw] \cup \Ddw
\end{equation}
is an SA-extension of $D$ containing $A$ (the saturation of $A$ by the complementary module~ $\Ddw$), and
$A + \Ddw =  B+ \Ddw$. Theorem \ref{thm:7.7} tells us that the pair $(B, \Ddw)$ has amalgamation in $V$. We arrive at the diagram
\begin{equation}\label{eq:10.5}
 \begin{array}{c}
\xymatrixrowsep{3mm}
\xymatrixcolsep{6mm}
    \xymatrix{ B \ar@{-}[d] \ar@{-}[rr] \ar@/_1.55pc/@{-}[ddd]_\SA
    & & A+ D^{\downarrow} = B+ D^{\downarrow}  \ar@{-}[dd]^\SA  \\
    A      \ar@{-}[rru] \ar@{-}[dd]^\SA & & \\
    & & \Ddw  \\
    D \ar@{-}[rru] & & \\
   }\end{array}
\end{equation}
for any SA-extension $D \subset A$ in $V$.
\pSkip

Before focusing on minimal $D$-cosets in $A+ \Ddw = B+ \Ddw$, we study $D$-cosets in an arbitrary submodule $U \supset D$ of $V$.

\begin{prop}\label{prop:10.3}
  Let $U \supset D $ be any submodule of $V$ containing $D$. The set of maximal elements in $U$ with respect to the ordering $\preceq_D$ is the set
  $$ \Fix_D(U) := \{ x\in U \ds | \forall d \in D: x+d = x \}$$
  of fix points in $U$ under the family of maps
  $z \mapsto z +d$, $U \to U$ with $d \in D$. Thus all minimal $D$-cosets in $U$ are singletons.
\end{prop}

\begin{proof}
  Of course, if $x + D$ is a singleton, then $x + D$ is a minimal $D$-coset. Conversely, $d + D \subset D$ for any $d \in D$. Let $x \in U$, whence $x + d + D \subset x + D$. If $x+D$ is minimal, this forces $x+d +D = x+D$.
We conclude by Remark \ref{rem:10.1} that $x + d = x$ for any $d \in D$.
\end{proof}

\begin{example}[The case of $U =D$]\label{exmp:10.4}
If $D$ has a maximal vector $\dmax$ in the $D$-ordering of $V$, then $\{ \dmax \}$
is the unique minimal $D$-coset in $D$. Otherwise $D$ does not contain any minimal $D$-coset.
\end{example}
We may study $D$-cosets in more general subsets of $V$ than submodules.

\begin{defn}\label{def:10.5}
A (nonempty) set $X \subset V$ is \textbf{stable under} $D$, if $X+D \subset X$, i.e., $x+d \in X$ for every $x \in X $ and $d \in D$. If this holds, we define
  $$ \Fix_D(X) := \{ x\in X \ds | \forall d \in D: x+d = x \}.$$

\end{defn}
\begin{prop}\label{prop:10.6} If a set $X \subset V$  is stable under $D$, then $\{ x\}$ is stable under $\Ddw$ for every $x \in \Fix_D(X)$, i.e., $x + \Ddw = \{x\}$.
\end{prop}
\begin{proof}
 Let $x \in \Fix_D(X)$ and $a \in \Ddw$. There exist $d_1,d_2 \in D$ with $a+ d_1 = d_2$. Thus
 $x +a = x + d_1 +a = x + d_2 = x$.
\end{proof}
\begin{prop}\label{prop:10.7}
$\Fix_D(V)$ is an ``ideal'' of the additive monoid $(V,+)$, i.e.,
$ V+ \Fix_D(V) \subset \Fix_D(V).$
\end{prop}
\begin{proof}
Let $v \in V$ and $x \in \Fix_D(V)$, then  $(x+v) + d = (x+d) + v = x+v$ for any $d \in D$,
\end{proof}

Using this chain of propositions  we determine the minimal $D$-cosets in $B + \Ddw$, and also get hold on some minimal $\Ddw$-cosets in $A + \Ddw = B + \Ddw$.
\begin{thm}\label{thm:10.8}
  Assume that $D$ is a submodule of an $R$-module $V$, $R$ any semiring, on which $\leq_D$ is an ordering. Let $D \subset A$ be any SA-extension in $V$ and $B = [(A\sm D) + \Ddw] \cup ~D$ the saturation of $A$ by the complementary module $\Ddw$ over $D$.  Recall that $(B, \Ddw) $ has amalgamation in $V$, $B \cap \Ddw = D$, and $A + \Ddw = B + \Ddw$ (Theorem \ref{thm:7.7}).

  The minimal $D$-cosets in $B$ are the singletons $\{ x \}$ with $x \in \Fix_D(A\sm D) $ and, in the case that $D$ has a maximal element, also $x = \dmax$. We have $x + \Ddw = \{x\}$ for all these vectors, and so the singletons $\{ x \}$ are also  minimal $\Ddw$-cosets in $A + \Ddw = B + \Ddw$.

  If $x \in \Fix_D(A) = \Fix_D(B)$, then $x + v \in \Fix_D(B)$ for every $v \in V$. Thus $\Fix_D(B)$ is an upper set under the quasiordering $\leq_V$, restricted to $B$.
\end{thm}

\begin{proof} a) Let $  v \in B + \Ddw$. Assume that $v + D$ is a minimal $D$-coset, and write $v = x + t$ with $x \in B$, $t \in \Ddw$. Then
$$
\begin{array}{ll}
  v + d & \subset v + d_1 + D = x + t + d_1 + D = x+d_2 + D.
\end{array}
$$
Due to the minimality of $v + D$ this forces
$$ v + D = x+d_2 + D \subset B+D.$$
More elaborately $v + D \in (A \sm D) + D = A \sm D$ or $v + D \subset \Ddw$. In the latter case $v + D$ is a minimal $D$-coset in $\Ddw$. This can only happen if $D$ has a maximal element $\dmax$, and then $v + D = \{ \dmax\}$ (cf. Example \ref{exmp:10.4}). In the former case $v + D = \{ x \}$ with $x\in A \sm D$. Thus the minimal $D$-cosets in $A + \Ddw = B + \Ddw$ are the singletons $\{ x\}$ with $x \in \Fix_D(A\sm D) $ and the singleton $\{ \dmax\}$, if $\dmax$ exists.
\pSkip
b)   Assuming now that $v + \Ddw $ is a minimal $\Ddw$-coset in $A + \Ddw = B + \Ddw$, we choose a minimal $D$-coset $u + D$ in the $D$-stable set $A + \Ddw$.  From $u + D \subset v + \Ddw$ we obtain
$u + \Ddw = u + D + \Ddw \subset v + \Ddw$, whence  $u + \Ddw = v + \Ddw$, by  minimality of $v + \Ddw$.  By a) we know that $ u \in \Fix_D(A) = \Fix_D(B)$, and conclude that $u + \Ddw = \{ u\}$ by Proposition \ref{prop:10.6}.
\pSkip
c) The last assertion in Theorem \ref{thm:10.8} is now clear by Proposition \ref{prop:10.6}.
\end{proof}

The question remains, which minimal $\Ddw$-cosets $v + \Ddw$ contain $D$-cosets and how many.  This question can be answered in a very general context. We assume that $V$ is an $R$-module over a semiring $R$, and consider for any submodule $E$ of $V$ the minimal $E$-coset $v+E$ in the set theoretic  sense, not assuming that $V$ is $E$-ordered.

\begin{lem}\label{lem:10.9}
If $v + E$ is minimal, then $v + E = u + E$ for every $u \in v +E$.
\end{lem}
\begin{proof}
 $u + E \subset v + E$. This forces $u + E = v+ E$, since $v + E$ is minimal.
\end{proof}
\begin{prop}\label{prop:10.10}
Let $D$ be any $R$-submodule of $V$. Assume that $v + \Ddw$ is a minimal $\Ddw$-coset in $V$. There exists at most one minimal $D$-coset $u + D \subset v + D$, and then
$u + \Ddw = v + \Ddw$.
\end{prop}
\begin{proof}
Suppose that $u_1 + D$ and $u_2 + D$ are minimal $D$-cosets contained in $v + \Ddw$. We have vectors $t_1, t_2  \in \Ddw$ for which $u_1 = v + t_1$, $u_2 = v + t_2$, and vectors $d_1, d_1', d_2, d_2'$ in ~ $D$ with  $t_1 + d_1 = d_1'$,  $t_2 + d_2 = d_2'$. Then $u_1 + d_1 = v + d_1'$,
$u_2 + d_2 = v + d_2'$, and so $u_1 + d_1 + d_2' = v + d_1' + d_2' = u_2 + d_1' + d_2$. We conclude by Lemma \ref{lem:10.9} that
$ u_1 + D = u_1 + d_1 + d_2' + D = u_2 + d_1' + d_2 + D = u_2 + D$. Furthermore
$u_1 + \Ddw = v + \Ddw$, again by Lemma \ref{lem:10.9}.
\end{proof}

\section{$D$-isolated vectors}\label{sec:11}
Given a module $V$ over a semiring $R$ and a submodule $D$ of $V$, we call a vector $v\in V$
\textbf{$D$-isolated}, if there exists neither a vector $d \in D$ with $v+d \neq v$ nor  vectors
$ x \neq v$ in $V$ and $d \in D$ such that $x + d =v$.
In other terms, $v + D = \{ v\}$, and $v \notin x + D$ for every $x \in V \sm D$.

We are interested in cases where $D$-isolated vectors show up in connection with a pair $(A,T)$ of submodules of $V$ with amalgamation, where $A \cap T = D$, $A$ an SA-extension of $D$ in $V$, and $T$ a $D$-complement of $A$ in $V$. First a simple but basic example.

\begin{examp}\label{exmp:11.1}
Let $R = \N_{\geq 0}$ and let $V$ be a totaly ordered set with a smallest element $0$.
We introduce the addition
$$ \lm_1 + \lm_2 = \max\{ \lm_1, \lm_2\}, \qquad \lm_1, \lm_2 \in V, $$
on $V$
and regard $V$ as $R$-module in the obvious way.
Assume that $D$ is a subset of $V$ containing~ $0$, and that  $D \sm \{ 0\}$ is convex in $V$.

Using the  notations from \S\ref{sec:9}, the set of vectors in $V$ without $D$-access is
$$ N = \Nac_D(D) = \{ \lm \in V \ds | \lm >D\},$$
and the complement of $N$ in $V$ is
 $$N^c = D^{\downarrow} = \{ x \in V \ds |  \exists d\in D : x \leq d\}. $$
  $A := N \cup  D$ is an SA-extension of $D$,
  since obviously $A \sm D$ is closed under addition and $(A \sm D) + D = A \sm D$ (cf. Proposition \ref{prop:7.3}).
  $T := D^{\downarrow}$ is a $D$-complement of $A$ in $V = A + D^{\downarrow}$ and  $$B:= [(A \sm D) + T] \ds \cup D = A.$$
Thus $A = B$ is a saturated SA-extension of $D$, and so we know by Theorem \ref{thm:7.7} that $(A, D^{\downarrow})$ has amalgamation in $V$.

We enquire, which elements $v$ of $V$ are $D$-isolated. Note that this is only of interest if there exist nonzero elements $\mu < D$.
\begin{enumerate}
   \eroman
\item Assume first $D \sm \00$ is not a singleton.
\begin{enumerate}
   \ealph
  \item If $v < D$, then $v +D = D$, and thus $v$ is not $D$-isolated.

  \item Let $v \in D$. If $D$ has a maximal element $\dmax$, then $v + \dmax \neq v$, and $v + \dmax = \dmax$ for all  $v < \dmax$. 
       If $D$ has no maximum, then $v + D \neq \{ v \} $ for any $v \in D$.
Thus $D$ is void of $D$-isolated vectors.
  \item Let $v > D$. Then $v + D = \{ v\} $, and if $x < v$ then $x + d = \max(x,d) < v$. Thus $v$ is $D$-isolated.
\end{enumerate}
We conclude: The $D$-isolated vectors of $V$ are the vectors $v > D$.
 \item If $D \sm \00 = \{ d \}$, then beside the vectors $v > d$ also  $v =d$ is $D$-isolated.
\end{enumerate}

\end{examp}

We introduce a special class of additive monoids, which will play a central role below.
\begin{defn}\label{def:11.2}
An additive monoid $(X,+)$ is \textbf{bipotent} (also called selective),  if for any $x_1, x_2 \in X$
\begin{equation}\label{eq:11.2}
  x_1 + x_2 \in \{x_1,x_2 \}.
\end{equation}
\end{defn}
For bipotent monoids we define a binary relation $\leq$ by
\begin{equation}\label{eq:11.2}
  x_1 \leq x_2 \dss \Leftrightarrow  x_1+x_2 = x_2.
\end{equation}
Clearly $x \leq x $ for any $x \in X$, and $x_1 \leq x_2$, $x_2 \leq x_1 $ implies $x_1 = x_2$.  If $x_1 \leq x_2 $ and $x_2 \leq x_3$, then
$x_1 + x_3 = x_1 + x_2+x_3 = x_2 + x_3 = x_3$. Furthermore $0 \leq x$ for all $x \in X$. Thus, the relation \eqref{eq:11.2} is a total ordering on the set $X$ with smallest element $0 =0_X$.

We compare this ordering $\leq$ with the quasiordering $\leq_X$ of $(X,+)$. If $x_1 \leq x_2$, then $x_1 + x_2 = x_2$, and so $x_1 \leq_X x_2$. Conversely, if $x_1 \leq_X x_2$, then there exists $y \in X$ such that $x_1 + y = x_2$, whence $x_1 \leq x_2$. Thus $\leq$ coincides with the ordering $\leq_X$ of $(x,+)$. This proves
\begin{prop}\label{prop:11.3} Every bipotent monoid $(X,+) $ is an upper bound monoid with total ordering $\leq_X$
 and smallest element $0 = 0_X$.
\end{prop}
This fact has a strong converse. Let $(X, \leq)$
be a totally ordered set with smallest element~ $0$. We define a composition
$X \times X \xrightarrow[ ]{+} X$ by the rule
\begin{equation}\label{eq:11.3}
  x_1 + x_2 = \max (x_1, x_2)
\end{equation}
(as done for $(V,+)$ in Example \ref{exmp:11.1}).
The composition $+$ is clearly commutative, and $0 + x = x$ for any $x \in X$. It can be proved by an easy straightforward way that the composition $+$ is associative \{e.g. check that
$(x_1 + x_2) + x_3 = x_1 + (x_2 + x_3)$ by going through the four cases  $x_1 \leq x_2 \leq x_3, x_1 \leq x_2 \geq x_3, \dots $ \}.
Thus $(X,+)$ is an additive monoid, whose zero element is the minimal element of $(X, \leq).$ This monoid is bipotent  due to \eqref{eq:11.3}.

Finally notice that the a map $\vrp: X \to X'$ form $X$ to a second bipotent monoid $(X', +)$ is a monoid homomorphism iff it respects the total orderings $\leq_X$ and $\leq_{X'}$ (in the weak sense, $x_1 \leq x_2 \Rightarrow \vrp(x_1) \leq \vrp(x_2)$), and maps $0_X$ to $0_{X'}$.
Thus we may state

\begin{prop}\label{prop:11.4}
  The category of bipotent monoids is canonically equivalent to the category of totally ordered sets with minimal element.
\end{prop}
In what follows we view bipotent monoids as the same objects as totally ordered sets with minimal element.

\begin{defn}\label{def:11.5}
$ $
\begin{enumerate} \ealph
  \item
A \textbf{bipotent retraction} of an additive monoid $(V,+) $ is a monoid homomorphism
$\vrp: (V,+) \to (X,+)$ to a bipotent submonoid $X$ of $V$ with $\vrp(x) = x$ for any $x\in X$.\footnote{We do not demand here that $(V,+)$ is upper bound, but in the construction below this will be the case.}

\item
We call such a retraction
 \textbf{special}, if for $v_1, v_2 \in V$ the following holds:
 $v_1 + v_2 \in \{ v_1, v_2 \} $ if $\vrp(v_1) \neq \vrp(v_2)$, while
 $v_1 + v_2 = \vrp(v_1) $ if $\vrp(v_1) =  \vrp(v_2)$.
\end{enumerate}

\end{defn}

Below we will use special retractions of $V$ to exhibit $D$-isolated vectors for  suitable submonoids $D$ of $V$.

\begin{rem}\label{rem:11.6} If $\vrp:(V,+) \to (X,+)$ is a special bipotent retraction, then for any $v_1, v_2 \in V$
\begin{equation}\label{eq:11.4}
  v_1 + v_2 = \left \{
  \begin{array}{ll}
   v_2 &  \text{if } \vrp(v_1) < \vrp(v_2), \\[1mm]
   v_1 &  \text{if } \vrp(v_1) > \vrp(v_2) ,\\[1mm]
   \vrp(v_1) &  \text{if } \vrp(v_1) =  \vrp(v_2), \\
  \end{array} \right.
\end{equation}
as follows from the fact that $\vrp$ respects addition, and so respects the quaiordering $\leq_V$ on $V$ and is restriction to the ordering $\leq_X$ on $X$.

It will turn out that for $\vrp:(V,+) \to (X,+)$  a special bipotent retraction the monoid $(V,+) $ is upper bound, but bipotent retractions of $(V,+)$ in general retains sense if $V$ is not upper bound and will  be useful also then, cf. Theorem \ref{thm:11.9} below.
\end{rem}

We obtain all special bipotent retractions by the following construction. Let $V$ be a \textbf{set} and $\vrp: V \to X$ a map to a subset $X$ of $V$ with $\vrp(x) =x $ for every $x \in X$ (a ``set theoretic retraction''). We choose a total ordering on $X$ with a minimal element $0 = 0_X$, and view ~$X$ as a bipotent monoid $(X,+)$, as explained above. Then we define a composition by the rule ~\eqref{eq:11.4} above.
\begin{thm}\label{thm:11.7} $ $
\begin{enumerate} \eroman
  \item $(V,+)$ is an upper bound additive monoid and $\vrp: V \to X$ is a special bipotent retraction.
  \item Every fiber $\ivrp(x)$ of $x$ is convex with respect to the partial ordering $\leq_V$, in particular $\ivrp(0) = \00$.
  \item Let $v \in V$, $x \in X$. If $\vrp(v ) \leq \vrp(x) = x$, then $v + x = x$.
If $\vrp(v ) > \vrp(x)$, then $v + x = v$.
\item $v + v = \vrp(v)$ for  every $v \in V$. Thus the bipotent submonoid $X$ of $(V,+)$ is uniquely determined by $(V,+)$. If $v_1 + v_1 = v_2 + v_2$, then  $v_1 + v_1 = v_1 + v_2$.
\end{enumerate}
\end{thm}
\begin{proof} a) It is obvious from rule \eqref{eq:11.4} that the composition $V \times V  \xrightarrow[ ]{+}  V$ is commutative and $v + 0 = v$ for all $v \in V$, where $0 = 0_X \in X$. For any $v_1, v_2, v_3 \in V$ it can be verified in a straightforward way that
  $(v_1 +  v_2) +  v_3 = v_1 + ( v_2 +  v_3)$ \{e.g. run  through all cases $\vrp(v_1) \ds \square \vrp(v_2) \ds \square \vrp(v_3)$ with $\square \in \{ < ,=, > \} $; some  cases can be settled simultaneously by interchanging $v_1$ and $v_3$\}. Thus
  $(V,+)$ is an additive monoid. Furthermore $(X,+)$ is a submonoid of $(V,+)$ by \eqref{eq:11.4}, and so $\vrp: V \to X$ is a special bipotent retraction of $V$.

  \pSkip
    b) We verify that $(V,+) $ is upper bound. Let $v, w_1,w_2 \in V$ and
    \begin{equation}\label{eq:str}
      v+ w_1 + w_2 = v. \tag{$*$}
    \end{equation}
    We need to prove that $v + w_1 = v$. Applying $\vrp $ to  $\eqref{eq:str}$ gives
    $$  \vrp(v)+ \vrp(w_1) + \vrp(w_2) = \vrp(v). $$
    Since $X$ is bipotent, we conclude that  $\vrp(w_1) \leq  \vrp(v)$
and $\vrp(w_2) \leq  \vrp(v)$. If $\vrp(w_1) <  \vrp(v)$, then $v + w_1 = v$.
If $\vrp(w_2) <  \vrp(v)$, then $v + w_2 = v$, and then by \eqref{eq:str}
again $v + w_1 =v$. If finally
 $\vrp(w_1)  = \vrp(w_2) = \vrp(v)$, then by \eqref{eq:str}  $v + w_1 = \vrp(v)$,
$v + w_1 +w_2 = \vrp(v)$, and we read off from  \eqref{eq:str} that $\vrp(v) = v$, $v + w_1 = v.$ Thus $v+ w_1 = v$ in all cases.

\pSkip c)
Claims (iii) and (iv) of Theorem \ref{thm:11.7} now follow by easy observations.
(iii) is clear by~ \eqref{eq:11.4}.
 If $v_1 \leq v \leq v_2$, then $\vrp(v_1) \leq \vrp(v) \leq \vrp(v_2)$, and so, if $\vrp(v_1) = \vrp(v_2)$ also $\vrp(v_1) = \vrp(v)$. Thus $\ivrp(\lm )$ is convex for any $\lm \in X$. (We use here in an essential way that~ $V$ is upper bound.) $\ivrp(0) = \00$ is obvious directly from \eqref{eq:11.4}. By \eqref{eq:11.4} also $\vrp(v) = v +v $. Finally, if $v_1 + v_1 =v_2 + v_2 $, then $\vrp(v_1)  = \vrp(v_2)$, and so $v_1 + v_1 = v_1 + v_2$, again by \eqref{eq:11.4}.
\end{proof}

\begin{examp}\label{exmp:11.7} Let $R = (R,\tG, \nu)$ be a supertropical semiring \cite{IzhakianRowen2007SuperTropical}, or more generally a $\nu$-semiring \cite{nualg}, where  $\tG$ is a bipotent subsemiring of $R$ and $\nu$ is a projection $R \to \tG$, i.e., $\nu(a) = a $ for every $a \in \tG$, satisfying $a + b = \nu(a)$ if $\nu(a) = \nu(b)$.   Then the projection  $\nu: R \to \tG$ is a bipotent retraction (Definition~ \ref{def:11.5}).

When $R$ is a supertrpical semifield \cite{zur05TropicalAlgebra,IzhakianRowen2007SuperTropical}, i.e., $\tT := R \sm \tG$ is an abelian group and the restriction $\nu|_{\tT}: \tT \to \tG$ is onto,  the subsemiring~ $\tG$ is totally ordered and $ a + b \in \{ a,b\}$ whenever $\nu(a) \neq \nu(b)$. So for  this case, the projection  $\nu: R \to \tG$ is a special bipotent retraction.

The familiar tropical (max-plus)  semifield $\T = (\R \cup \{ -\infty \}, \max, +)$ is a biopotent semifield. It extends to a supertropical semifield $F = (\R \cup \{-\infty\} \cup \R^\nu, +, \cdot \,)$, where $\R^\nu$ is a second copy of $\R$, in which $\T$ embeds in $F$ as its  ghost ideal $\tG$. Then the projection $\nu: F \to \T$ is  a special bipotent retraction.
\end{examp}

We now  assume only that $\vrp:(V,+) \to (X,+)$ is a bipotent retraction, $V$ not necessarily upper bound. We choose a  subset $\brD$ of $X$ that contains $0$, for which $\brD \sm \00$ is convex in $X$. Then
\begin{equation}\label{eq:11.5}
  D : = \ivrp(\brD)
\end{equation}
is a submonoid  if $V$, for which the following holds:
\begin{equation}\label{eq:11.6}
  \text{If $d_1 \leq_V v \leq_V d_2$ where  $\vrp(d_1)$ and $\vrp(d_2)$ are in $\brD \sm \00$, then $v \in D$. }
\end{equation}
Let $$\Ddw := \{ v \in V \ds | v \leq_V d \text{ for all } d \in D\}.$$

\begin{lem}\label{lem:11.9} $\Ddw  = \ivrp(\brD^{\downarrow})$.
\end{lem}

\begin{proof} $(\subseteq)$: Let $v \in \Ddw$, i.e., $v+w = d$ with $w \in V$, $d \in D$. Then $\vrp(v) + \vrp(w) = \brd$ with $\brd := \vrp(d) \in \brD$. Thus $\vrp(v) \in \brD^{\downarrow}$.

 \pSkip  $(\supseteq)$: Let $v \in V$, $\vrp(v) \leq \brd$. There exists $w \in V$ with
 $\vrp(v+w) = \vrp(v) + \vrp(w) = \brd$, and so $v+w \in \ivrp(\brD) = D$.
\end{proof}

\begin{thm}\label{thm:11.9}
 Assume that $\lm \in X$ is $\brD$-isolated. Then every $v \in \ivrp(\lm)$ is $D$-isolated.
\end{thm}

\begin{proof}
Let $\vrp(v) = \lm$. Then for any $d \in D$ $\vrp(v+d) = \vrp(v) + \vrp(d) = \vrp(v) = \lm$. If $u + d = v$, then $\vrp(u) + \vrp(d) = \vrp(v) = \lm$.
Thus $\vrp(u) = \lm$, and so  $u + d = u$, as proved. We conclude that $u= v$.
\end{proof}

Thus, what had been observed about $D$-isolated vectors in Example \ref{exmp:11.1}, remains valid in the present more general setting mutatis mutandis.

\section{SA-submodules induced by actions on upper bound   modules }\label{sec:12}

Let $(V,+) $ and $(X,+)$ be additive monoids. An \textbf{action} of $V$ on $X$ is a map
$\al: V \times X \to X$ having the following properties:  Write
$\al(u,x) = u \pal x$
$$ u \pal (x_1 + x_2 ) \ds = (u \pal x_1) + x_2 \quad \text{(and so also  } = x_1 + (u \pal  x_2 )) ,  $$
$$(u_1 + u_2) \pal x = u_1 \pal (u_2 \pal x), \qquad 0  \pal x = x. $$
In this situation,  for any $x \in S$ define
\begin{equation}\label{eq:12.1}
  \cal(s) := \{ u \in V \ds| u \pal s = s\}.
\end{equation}
It is immediate that $\cal(s)$ is a submonoid of $V$: If $u_1 \pal s = s$ and $u_2 \pal s = s$, then
$$ (u_1 + u_2) \pal s = u_1 \pal (  u_2  \pal s )  = u_1  \pal s = s, \quad \text{and } 0 \pal s = s.$$
We define for $u \in V$:
$$ \tlu := u \pal 0_X \in X. $$
Note that for $u_1 + u_2 \in V$ we have
$$ \begin{array}{ll}
\widetilde{u_1 + u_2} & = (u_1 + u_2) \pal 0 =  u_1 \pal (u_2 \pal 0) \\[2mm] & = u_1 \pal \tlu_2  = u_1 \pal (0 + \tlu_2) = (u_1 \pal  0) + \tlu_2 = \tlu_1 + \tlu_2.   \end{array}
$$
Furthermore, $\tl0 = 0 \pal 0_X = 0_X$.
Thus the map  $u \mapsto \tlu$ is a homomorphism of the monoid $V$ onto a submonoid of $X$.

From now on \emph{we assume that the monoid $(X,+)$ is upper bound}.
\begin{rem}\label{rem:12.1}
  For any $u \in V$, $x \in X$,
  $$ u \pal x = u \pal (0 + x) = (u \pal 0) + x = \tlu + x, $$
and so $u \pal x \geq_X x.$
\end{rem}
\begin{prop}\label{prop:12.2}   $\cal(s)$ is an SA-submonoid of $V$ for any $s \in X$.

\end{prop}
\begin{proof}
  Let $u_1 + u_2 \in \cal(s)$. Then by Remark \ref{rem:12.1},
  $$\begin{array}{ll}
       s = (u_1 + u_2) \pal s = u_1 \pal (u_2 \pal s)  \geq_X u_2 \pal s \geq_X s,
    \end{array}$$
  and thus $u_2 \pal s = s $ (and so also $u_1 \pal s = s$).
\end{proof}

\begin{prop}\label{prop:12.3} If $s_1, s_2 \in X$, then
$\cal(s_1) + \cal(s_2) \subset \cal(s_1 + s_2)$. Thus
$$ s \leq_X t \dss \Rightarrow \cal(s) \subset \cal(t).$$
\end{prop}
\begin{proof}
Let $u \in \cal(s_1)$. Then
$$ u + (s_1 + s_2) = (u + s_1 ) + s_2 = s_1 + s_2.$$
Thus $u \in \cal(s_1  + s_2)$, implying that $\cal(s_1) \subset \cal(s_1 + s_2)$.
Similarly,  $\cal(s_2) \subset \cal(s_1 + s_2)$, and so $\cal(s_1) + \cal (s_2)\subset \cal(s_1 + s_2)$.
\end{proof}

We further define
\begin{equation}\label{eq:12.2}
  \cal(S) := \bigcup_{s\in S} \cal(s) = \{ u \in V \ds | \exists s \in S : u \pal s = s\}.
\end{equation}
for any (nonempty) subset $S$ of $X$.
\begin{prop}\label{prop:12.4} $ $
\begin{enumerate} \ealph
\item
If $S$ is closed under addition then $\cal(S)$ is an SA-submodule of $V$.
  \item If $T$ is a second subset of $X$, closed under addition, then
  $$\cal(S) + \cal(T) \subset \cal(S+T). $$
  \item If $S$ and $T$ are cofinal subsets of $X$, i.e., for every $s \in S$ there is some $t \in T$ with $s \leq_X t$, and vice versa, then $\cal(S) = \cal(T)$.

\end{enumerate}
\end{prop}
\begin{proof}
We infer from the definition \eqref{eq:12.1} of $\cal(S)$ and Proposition \ref{prop:12.2} that $\cal(S)$ is a submodule of $V$. By this proposition also claims (b) and (c) are evident. Given $u_1,u_2 \in V$ with $u_1 + u_2 \in \cal(S)$, there is some $s \in S$ for which  $u_1 + u_2 \in \cal(s)$. By Proposition \ref{prop:12.3} we conclude  that $u_1, u_2 \in \cal(s) \subset \cal(S).$ Thus the submodule $\cal(S)$ is SA in $V$.
\end{proof}

\begin{example}\label{exmp:12.5} Given $x \in V \sm \00 $, i.e., $x > 0$, let
$$ S = \{ nx \ds | n \in \N \}. $$
 We have $x \in C(nx) $ iff $nx= (n+1) x$, and thus  meet the following dichotomy: If there exists
 $n \in \N$ with $nx = (n+1)x$, then $C(S) = C(nx)$ for the smallest such number $n$. Otherwise ~$C(S)$ is the union of the submonoids
 $C(x ) \subsetneqq C(2x) \subsetneqq C(3x) \subsetneqq \cdots $. In this case we write $C(S) = C_\om(x)$.
\end{example}
We turn to the primordial example for an action of $V$ on $X$. Here $X = V$ and $\al$ is given by
$v \pal x := v + x$. We denote the monoids $\cal(x)$ simply by $C(x)$, and then  for any $x \in V$ have
\begin{equation}\label{eq:12.3}
  C(x) = \{ v \in V \ds| v+x = x\}.
\end{equation}
Recall that each of these sets $C(x)$ is an SA-submodule of $V$, and for any $x,y \in V$ the submonoid $C(x+y)$ contains $C(x)$ and $C(y)$, whence
\begin{equation}\label{eq:12.4}
  x \leq_V y \dss\Rightarrow C(x) \subset C(y).
\end{equation}
Turning to the SA-submonoids
$C(S) := \bigcup_{s\in S} C(s)$, for $S$ a subset of $V$, closed under addition, the following case deserves special interest. Writing $\leq $ instead of $\leq_V$, for short, an archimedean feature comes into sight.
\begin{defn}\label{def:12.6}
The archimedean class of any $x \in V $ is
$$ \Arch(x):= \{ y \ds | \exists n: y \leq nx, \exists m: x \leq my \} .$$
\end{defn}
The following is easily seen and certainly very well known.

\begin{rem}\label{rem:12.7}
  Let $x,y \in V \sm \00$ be in the  same archimedean class. Then either both sequences
  $\{ mx \ds | m \in \N \} $, $\{ ny \ds | n \in \N \} $ become constant of values
  $m_1 x$, $n_1 y$, and $\Arch(m_1 x) = \Arch (n_1 y)$, or $C_\om(x) = C_\om(y).$ Of course, $\Arch(0) = \00$.
\end{rem}

In any upper bound monoid $(X,+)$ we define archimedean classes as above in Definition \ref{def:12.6}, here with the ordering $\leq_X$ on $X$.

Assume now that an action $\al: V \times X \to X$ is given, where both $V$ and $X$ are modules over a semiring $R \neq \00$, and, as before, $X$ is upper bound. Then, for any $x \neq 0$ in $X$, the set
\begin{equation}\label{eq:12.5}
  R_x := \{ \lm \in R \ds | \lm \cdot C(x) \subset C(x)\}
\end{equation}
is a subsemiring of $R$, as is immediate from the fact that $C(x)$ is a submonoid of $V$. Consequently,
\begin{equation}\label{eq:12.6}
  \N_0 \cdot 1_R \subset  R_x,
\end{equation}
since $ \N_0 \cdot 1_R$ is the unique smallest subsemiring of $R$.

A subset $Y$ of $R$
is called \textbf{convex}
in $R$ (with respect to the quasiordering $\leq_R$) if for $\lm_1 \leq \lm_2$ in $Y$ every $\lm \in R$ with $\lm_1 \leq \lm \leq  \lm_2$ is also in $R$.
\begin{prop}\label{prop:12.8}
For every $x \in X$ the semiring $R_x$ is convex in $R$.
\end{prop}
\begin{proof}
  Let $\lm_1, \lm_2 \in R_x$ with $\lm_1\leq_R \lm \leq_R \lm_2$, i.e., $\lm_1 + \mu = \lm$, $\lm + \nu = \lm_2$ for some $\mu, \nu \in R$. Given $u \in C(x)$, we have
  $\lm_1 u + \mu u = \lm u$, $\lm u + \nu u= \lm_2u$. Thus
  $$
  \begin{array}{ll}
    x & = (\lm_1 u) \pal x \leq (\lm_1 u) \pal x + (\mu u) \pal x  \\[1mm]
    & = (\lm u) \pal x \leq (\lm u) \pal x +(\nu u) \pal x \\[1mm] & = (\lm_2 u) \pal x =x.
  \end{array}
  $$
  This implies that
$(\lm u) \pal x =x$, and proves that  $\lm u \in C(x)$.
\end{proof}
\begin{defn}\label{def:12.8}
Let $\mfo_R$ denote the convex hull of $\N_0 \cdot 1_R$ in the semiring $R$ with respect to ~ $\leq_R$. It is the smallest SA-subsemiring of $R$, cf. \cite[\S5]{Dec}.
\end{defn}
\begin{cor}\label{cor:12.10}
  For any subset $S$ of $X$, closed under addition, $\cal(x)$ is an $\mfo_R$-submodule of $V$.
\end{cor}
\begin{proof}
  For any $s \in S$ $R_s$ is a convex subsemiring of $R$ containing $\N_0 \cdot 1_R$, whence containing~ $\mfo_R$. Thus $\cal(s)$ is an $\mfo_R$-submodule of $V$. The same holds for the union (=sum) $\cal(S)$ of the $\cal(s)$, $s \in S$,  cf. \eqref{eq:12.2}.
\end{proof}

\section{Amalgamation in the category of upper bound monoids}\label{sec:13}

The amalgamation theory for submodules of an additive monoid $(V,+)$, as developed in \S\ref{sec:1}-\S\ref{sec:5}, can be amended in a natural way to an amalgamation of upper-bound monoids, since there is a canonical reflection $V \to \brV$ from the former to the latter category (cf. \cite[\S5]{IKR1}).

$(\brV,+)$ arises from $(V,+)$ by dividing out the natural congruence  relation $\equiv_V$, which  turns the quasiordering  $\leq_V$ on $V$ to a (partial) ordering $\leq_{\brV}$ on $\brV = V / \equiv_V$. This congruence is  given by
$$ x \equiv_V y \dss\Leftrightarrow x \leq_V y \text{ and } y \leq_V x.$$
We denote the congruence class of a vector $x \in V$ by $\brx$, and then have for $x,y \in V$ the explicit description
\begin{equation}\label{eq:13.1}
  \brx = \bry \dss\Leftrightarrow \exists z,w \in V: x+z = y, y+w = x.
\end{equation}
We name $\brV$ the \textbf{upper bound  monoid associated to $V$}.

If $V$ is a module over a semiring $R$, then it is immediate that $\brV$ is a module over the upper bound semiring $\brR = R / \equiv$, with scalar mutiplication given by
\begin{equation}\label{eq:13.2}
  \bra \bar v = \overline{av}
\end{equation}
for $a \in R $, $v \in V$. Then consequently, we say that the $\brV$ is the upper bound  $\brR$-\textbf{module associated to $V$}.

Below we most of the time work in the category of $R$-modules for $R$ an upper bound  semiring. Monoids  can be subsumed here  by taking
$R = \brR =\N_0$.

As common, we say that a subset $S \subset V $ is \textbf{convex}, if $s \in S$  for any $s_1, s_2 \in S$, $s\in V$ with $s_1 \leq_V s \leq_V s_2$.\footnote{In the special case $V=R$  already   defined in \S\ref{sec:12}.}
We cite the following useful fact, valid in any monoid $(V,+)$.

\begin{prop}[{\cite[Proposition 5.7]{IKR1}}]\label{prop:13.1}
  A submodule $S$ of $V$ is SA in $V$ iff $S$ is a union of congruence classes and $\brS= S/ \equiv_V $ is SA in $\brV$.
\end{prop}

In \S\ref{sec:12} we started a study of the SA-submodule $C(x) = \{ v \in V \ds | v+x\} $ of $V$ for every $x \in V$ in the case that $V$ is an upper bound monoid. We now continue this study for  the upper bound monoid $\brV = V / \equiv_V$ associated to \emph{any} additive monoid $(V,+)$, but instead of arguing in $\brV$ and then passing to $V$ by the use of Proposition \ref{prop:13.1}, we work directly in $V$ by using the quasiordering $\leq_V$.  For a given $x \in V$ we define
\begin{equation}\label{eq:13.3}
  \brC(x) := \{ u \in V \ds | u +x \leq_V x\}.
\end{equation}
Since always $x \leq_V u +x$, this means that
\begin{equation}\label{eq:13.4}
  \brC(x) = \{ u \in V \ds | u +x \equiv_V x\},
\end{equation}
and thus, if $V$ happens to be upper bound, $\brC(x) = C(x)$.
More generally we may assume that $V$ is an $R$-module, $R$ any semiring. Then $\brV$ is an $\brR$-module for $\brR = R/ \equiv_R$. Let
$$\mfo_R = \conv(\N_0\cdot 1_R)$$ be
the convex hull of $\N_0 \cdot 1_R$ in $R $ with respect to $\leq_R$. Then we conclude by Corollary \ref{cor:12.10} that $\brC(x)$ is an $\mfo_R$-submodule of $V$. We often write $\leq, \equiv, \dots,$ for $\leq_V, \equiv_V, \dots$, when  the ambient monoid $V$ is clear from the context.
\begin{prop}\label{prop:13.2}
$\brC(x)$ is an SA-submonoid  of $V$.
\end{prop}
\begin{proof}
If $u_1 + x \leq x$, $u_2 + x \leq x$, then $u_1 + u_2 + x \leq u_1 + x \leq x$. Thus $\brC(x)$ is a submonoid. Conversely, if $u_1 + u_2 + x \leq x$, then $u_1 + x \leq u_1 + u_2 + x \leq x$, and also $u_2 + x \leq x$. Thus $\brC(x)$ is SA in $V$.
\end{proof}
\begin{prop}\label{prop:13.3} If $x \leq_V x'$, then $\brC(x) \subset \brC(x')$. Consequently, if  $x \equiv_V x'$, then $\brC(x) = \brC(x')$.
\end{prop}
\begin{proof}
If $x \leq_V x'$, then $x' = x + y$ for some $y \in V$. Thus $u + x \leq_V x$ implies
$u +x + y \leq_V x + y$, i.e., $u + x' \leq_V x'$.
\end{proof}

\begin{defn}\label{def:13.4}
For any $x \in V$ we introduce the subset
$$ \brC_\om(x) := \bigcup_{n \in \N} \brC(nx)$$
of $V$.
\end{defn}
Since by Proposition \ref{prop:13.2}
\begin{equation}\label{eq:13.5}
   \brC( n x)  \subset \brC(nx + x) = \brC((n+1)x)
\end{equation}
it is clear, that $ \brC_\om(x) $ is again an SA-submodule of $V$.
Consequently, in the case that $V$ is upper bound  we define
$$ C_\om(x) := \bigcup_{n \in \N} C(nx), $$
which extends the notation in Example \ref{exmp:12.5} to all $x \in V$.
\begin{defn}\label{def:13.5} The \textbf{archimedean class} $\Arch_V(x)$ of an element $x \in V$ is the set of all $y \in V$ such that $x \leq_V n y$ and $y \leq_V m x $ for some $n,m \in \N$.
\end{defn}

\begin{lem}\label{lem:13.6}
Let $x,y \in V$, and assume that $x \leq_V my$ for some $m \in \N$.
Then $\brC(x) \subset \brC(my)$ and $\brC_\om(x) \subset \brC_\om(y)$.
\end{lem}
\begin{proof} For every $k\in \N$ we have $kx \leq_V k my$, whence  $\brC(kx) \subset \brC(kmy)$ by  Proposition \ref{prop:13.3}. This gives both claims.
\end{proof}

The following in now evident.
\begin{prop}\label{prop:13.7}
  If $\Arch_V(x) = \Arch_V(y)$, then $\brC_\om(x) = \brC_\om(y)$.
\end{prop}

More generally, given a subset $S$ of $V$ with $S +S \subset S$, the subset
\begin{equation}\label{eq:13.6}
   \brC(S) := \bigcup_{s \in S} \brC(s)
\end{equation}
of $V$ is an SA-submodule, since for $x \in \brC(s)$, $ y \in \brC(t)$ we have
$$ x + y \in  \brC(s) +  \brC(t) \subset  \brC(s + t)$$
(cf. Proposition \ref{prop:13.3}), and $\brC(s+t)$ is an SA-submodule of $V$. Again $\brC(S)$ is an $\mfo_R$-submodule of $V$ in the case that $V$ is an $R$-module. For any $x \in V$
\begin{equation}\label{eq:13.7}
   \brC_\om(x) =  \brC( \N x).
\end{equation}

Given two subsets $S$ and $T$ of $V$, closed under addition, suppose that for any $t \in T$ there is some $s \in S$ with $t \leq s$. Then $\brC(T) \subset \brC(S)$. It follows that $\brC(T) = \brC(S)$, if $S$ and $T$ are cofinal under $\leq_V$.

We are ready to construct an additive monoid which is the amalgamation of submonoids $A_1, \dots, A_r$, in which for any tuple $(x_1, \dots, x_r) \in A_1 \times \dots \times A_r$
the family $(\brC_\om(x_1), \dots, \brC_\om(x_r) )$ has amalgamation in $V$ and
$\brC_\om(x_1)+  \cdots +  \brC_\om(x_r)$ is SA in $V$.

Starting with finitely many submonoids $A_1, \dots, A_r$ of an additive monoid $V_0$, we introduce the   amalgamation (cf. \S\ref{sec:4})
\begin{equation}\label{eq:13.8}
  V = A_1 \infty_{V_0} \cdots \infty_{V_0} A_r.
\end{equation}
We identify each $A_k$ with the submonoid $j_k(A_k)$ of $V$, as explained in \S\ref{sec:4}.
Then
\begin{equation}\label{eq:13.9}
  V = A_1 + \cdots + A_r.
\end{equation}
Here two tuples $(a_1, \dots, a_r)$, $(b_1, \dots, b_r) $ in $A_1 \times \dots \times A_r$ with the same sum $a_1 + \cdots +  a_r = b_1 +  \cdots + b_r$ are exchange equivalent in $V$.
Given $x = x_1 + \cdots + x_r$, $x_k \in A_k$, we define the submonoid
\begin{equation}\label{eq:13.10}
  W = W(x) := \{ y \in V \ds | \exists n \in \N : y \leq_V n x\}.
\end{equation}
 It is the convex hull of $\N x$ in $V$ (with respect to $\leq_V$). For any $z \in W$
 $$\begin{array}{ll}
   \brC_\om(x) & = \{ u \in V \ds | \exists n \in \N : u \leq_V n z \}  =
 \{ u \in W \ds | \exists n \in \N : u \leq_V n z \},
   \end{array}
   $$
 since $W$ is convex (=SA) in $V$, and so $\brC_\om(z)$ is SA in $W$. In particular, each $\brC_\om(x_k)$ is SA in $W$, and
\begin{equation}\label{eq:13.11}
\brC_\om(x) = \brC_\om(x_1) \infty_V \ds\cdots \infty_V \brC_\om(x_r) \underset{\SA}{\subset} W,
\end{equation}
 as follows from Theorem \ref{thm:4.6}.

Furthermore, if $x' = x_1' + \cdots + x_r'$ is a vector with $x_k' \in A_k$, and $x \leq x'$, then we infer from  ~\eqref{eq:13.1} that $W(x) \subset W(x').$ Thus we meet a hierarchy of amalgamated SA-submodules of $V$,
\begin{equation}\label{eq:13.12}
 \brC_\om(x_1) \infty \ds \cdots \infty \brC_\om(x_r) \ds  \subset \brC_\om(x_1') \infty \ds \cdots \infty \brC_\om(x_r').
\end{equation}
This construction can be enlarged by choosing instead of the vectors $x_k \in A_k$ subsets $S_k$ of~ $A_k$ which are closed under addition. In a complete analogy to  arguments above we obtain the following.

\begin{thm}\label{thm:13.8}
Take $ V = A_1 \infty \cdots \infty = A_1 + \cdots + A_r$ as above (cf. \eqref{eq:13.8}, \eqref{eq:13.9}). For every $k \in \{ 1, \dots, k \} $ let $S_k$ be a subset of $A_k$ closed under addition, and furthermore
$$ S := S_1 + \cdots + S_r \subset A_1 + \cdots + A_r =V.$$ Then
$$ \brC(S_k) = \{ z  \in A_k \ds | \exists s \in S_k : z+ s \leq_V s \}  $$
is an SA-submonoid of $A_k$,
$$ \brC(S) = \{ z  \in V \ds | \exists s \in S_k : z+ s \leq_V s \}  $$
is an SA-submonoid of $A_k$, and $\brC(S)$ is the amalgamation of the submonoids $\brC(S_1), \dots, \brC(S_r)$ in $V$. If the $A_k$ are $R$-modules, $R$ any semiring, then all $\brC(S_k)$ and $\brC(S)$ are $\mfo_R$-submodules of $V$.
\end{thm}

\end{document}